\documentclass[10pt,a4paper]{amsart}
\usepackage[english]{babel}
\usepackage{srcltx}
\usepackage[usenames]{color}
\usepackage{amsthm}
\usepackage{amsfonts, amssymb, amsmath, amscd, latexsym}
\usepackage{mathrsfs,epsfig}
\usepackage[latin1]{inputenc}
\swapnumbers
\definecolor{labelkey}{rgb}{0,0,1}
\definecolor{refkey}{rgb}{1,1,1}

\newtheorem{thm}{Theorem}[section]
\newtheorem{lemma}[thm]{Lemma}
\newtheorem{cor}[thm]{Corollary}
\newtheorem{prop}[thm]{Proposition}

\theoremstyle{remark}
\newtheorem{remark}[thm]{Remark}

\newtheorem{defin}[thm]{Definition}
\numberwithin{equation}{section}

\newcommand{\RR}{{\mathbb R}}
\newcommand{\lir}{{\lim_{r \to +\infty}}}
\newcommand{\liro}{{\lim_{r \to 0}}}

\def\la{\lambda}
\def\R{\boldsymbol{R}}
\def\Q{\boldsymbol{Q}}
\def\eu{\textrm{e}}
\def\ep{\varepsilon}

\newcommand{\supp}{\mathop\mathrm{supp}\nolimits}

\author{Luca Bisconti and Matteo Franca} \address[L.\
Bisconti]{Dipartimento di Matematica e Informatica, Universit\`a degli
  Studi di Firenze, Via S.\ Marta 3, I-50139 Firenze, Italy. Partially
supported by G.N.A.M.P.A. - INdAM (Italy).}
\address[M.\ Franca]{Dipartimento di Ingegneria Industriale e Scienze
  Matematiche,
  Universit\`a Politecnica delle Marche, Via Brecce Bianche, I-60131
  Ancona, Italy. Partially
supported by G.N.A.M.P.A. - INdAM (Italy)}

\begin{document}

\title{On a non-homogeneous and non-linear heat equation.}

\begin{abstract}
  We consider the Cauchy-problem for a parabolic equation of the
  following type:
  \begin{align*}
    \frac{\partial u}{\partial t}= \Delta u+ f(u,|x|), 
  \end{align*}
  where $f=f(u,|x|)$ is supercritical. We supply this equation by the initial
  condition $u(x,0)=\phi$, and we allow $\phi$ to be either bounded or
  unbounded in the origin but smaller than stationary singular
  solutions. We discuss local existence and long time behaviour for
  the solutions $u(t,x;\phi)$ for a wide class of non-homogeneous
  non-linearities $f$.  We show that in the supercritical case, Ground
  States with slow decay lie on the threshold between blowing up
  initial data and the basin of attraction of the null solution.  Our
  results extend previous ones allowing Matukuma-type potential and
  more generic dependence on $u$.

  Then, we further explore such a threshold in the subcritical case
  too. We find two families of initial data $\zeta(x)$ and $\psi(x)$
  which are respectively above and below the threshold, and have
  arbitrarily small distance in $L^{\infty}$ norm, whose existence is
  new even for $f(u,r)=u^{q-1}$.  Quite surprisingly both $\zeta(x)$
  and $\psi(x)$ have fast decay (i.e. $\sim |x|^{2-n}$), while the
  expected critical asymptotic behavior is slow decay (i.e. $\sim
  |x|^{2/q-2}$).
\end{abstract}
\maketitle

 stability.\\
\textbf{MSC (2010)}: $\,$ 35k80, 35b60, 35b40 \vspace{5mm}

\pagestyle{myheadings} \markboth{}{On a parabolic equation...}

\section{Introduction}
The purpose of this paper is to study the asymptotic behavior of
positive solutions of the following Cauchy problem
\begin{align}
  &\frac{\partial u}{\partial t}= \Delta u+  f(u,|x|), \label{parab}  \\
  &u(x,0) = \phi(x), \label{data}
\end{align}
where $x \in \RR^n$, $n >2$, and  $f=f(u, |x|)$ is a potential which is null for
$u=0$.

In the last 20 years this problem has raised a great interest,
starting from the model cases $f(u,|x|)=u^{q-1}$ and
$f(u,|x|)=|x|^{\delta}u^{q-1}$.  Due to symmetry
  reasons, from now on we use notations which are standard for the
stationary problem, so we refer to $f(u,|x|)=u^{q-1}$ as the
model case.  In so doing it will be more clear the relationship between the
critical  values for \eqref{laplace} appearing below, and their
meaning in other contexts of functional analysis.

We assume that $f$ is \emph{supercritical} with respect to the Serrin
critical exponent, i.e. $2_*:=\frac{2(n-1)}{n-2}$, and for some specific
results we require $f$ to be \emph{supercritical} also to the Sobolev
critical exponent, i.e. $2^*:=\frac{2n}{n-2}$. The exponents $2_\ast$
and $2^\ast$ are related to
the continuity of the trace operator in $L^q$ and to the possibility
to embed $H^1$ in $L^q$, respectively.

Here, we want to analyze the structure of the border of the
basin of attraction to the null solutions,
and the set of initial data $\phi$ of solutions of
\eqref{parab}--\eqref{data} which blow up in finite time.
Our main aim is to extend the discussion to a wide class of
potentials: For the remainder of the paper we will always assume the following
\begin{itemize}
\item[$\boldsymbol{F0}$:] The function $f(u,r)$ is locally Lipschitz
  in $u$ and $r$ for any $u \ge 0$ and $r>0$. Moreover $f(0,r) \equiv
  0$, $f(u,r)>0$ and $f(u,r)$ is increasing in $u$, for any $u>0$ and
  any $r>0$, and there is a constant $C(u)>0$ such that $f(u,r)r^2 \le
  C(u)$ for $0<r \le 1$.
\end{itemize}
\noindent Further hypotheses on $f$ will be given in the sequel (see
conditions $\boldsymbol{G0}$, $\boldsymbol{Gu}$,
and $\boldsymbol{Gs}$ in Section \ref{steadyasympt}). Possible
examples are the following
\begin{align}
  & f(u, |x|) = k_1(|x|)|u|^{q_1-1} ,   \label{eq:potential-0}\\
  & f(u, |x|) = k_1(|x|)|u|^{q_1-1} +
  k_2(|x|)|u|^{q_2-1}, 
  \label{eq:potential-1}\\
  & f(u, |x|) = k_1(|x|) \min \{ u^{q_1-1}, u^{q_2-1}
  \}, \label{eq:potential-2.5}
\end{align}
where $q_1<q_2$ and $k_i=k_i(|x|)$, $i= 1,2$, are supposed non-negative and
Lipschitz continuous, and such that
\begin{equation}\label{asintotico}
  k_i(r) \sim A_i r^{\delta_i} \,  \; \textrm{ as $r\to 0$,} \qquad
  k_i(r) \sim B_i r^{\eta_i} \,  \; \textrm{ as $r\to +\infty$,}
\end{equation}
where $A_i,B_i \ge 0$, $\sum_i A_i >0$, $\sum_i B_i >0$, $q_i>2$ and
$\delta_i,\eta_i > (n-2)q-(n-1)$, for $i=1,2$ (so for
$\delta_i=\eta_i=0$ we require $q_i>2_\ast$).  More precise
requirements on $k_i$, $i=1,2$, will be provided later on
according to the ones on $f$.

Due to the nature of the considered potentials, in general we cannot
expect the solutions of (\ref{parab})--(\ref{data}) to be
differentiable, or even continuous, everywhere. In
  fact, we deal also with solutions that may be not defined at $x=0$
  since they become unbounded.

 In Section~\ref{sec:local-existence}
we prove the existence of a proper class of weak solutions to the
considered problem (see Lemma~\ref{lemma-M} and
Theorem~\ref{teo:CB-CS-mild-sol} below) and we actually show their
improved properties.
We consider the classes of $C_B$-mild and $C_S$-mild solutions to
(\ref{parab})--(\ref{data}) (see the definitions \ref{def:CB-sol} and
\ref{def:CS-sol} below, see also \cite{W}) proving local and global
existence as well as uniqueness. \smallskip

 Let $u(x,t;\phi)$ be the solution of
(\ref{parab})--(\ref{data}). The analysis of the long time behavior of
$u(x,t;\phi)$ is strongly based on the separation properties of the
stationary solutions of (\ref{parab}), i.e. functions $u(x)$ solving
\begin{equation}\label{laplace}
  \Delta u+ f(u,|x|)=0 \, ,
\end{equation}
and in particular on the properties of radial solutions. Notice that
if $u(x)$ is a radial solutions of (\ref{laplace}), setting
$U(r)=u(x)$, for $r=|x|$, then  $U=U(r)$ solves
\begin{equation}\label{radsta}
  U''+\frac{n-1}{r} U'+ f(U,r)=0 \, ,
\end{equation}
where $``\,'\, "$ denotes the derivative with respect to $r$.

In the whole paper we use the following notation: $U(r)$
is regular if $U(0)=\alpha>0$,  so we set $U(r)=U(r,\alpha)$, and we
say that $U(r)$ has a \emph{non-removable singularity} (or shortly
that it is \emph{singular}) if $\liro U(r)=+\infty$.  Similarly, we say
that a positive solution $V(r)$ of (\ref{radsta}) has fast decay
(f.d.) if $\lir V(r) r^{n-2}=\beta>0$ and we set $V(r)=V(r,\beta)$, and
that $V(r)$  has slow decay (s.d.) if $\lir V(r) r^{n-2}=+\infty$.

Further, $U(r)$ is a ground state (G.S.) if it is a regular solution of
(\ref{radsta}) which is positive for any $r>0$.  Instead, we say that $U(r)$ is a
singular ground state (S.G.S.) if it is a singular solution of
(\ref{radsta}) which is positive for any $r>0$.  The asymptotic
behavior of singular and slow decay solution is well understood and
will be discussed in more details in Section~\ref{steadyasympt}.

Roughly speaking, the $\omega$-limit set of \eqref{parab} is (usually)
made up by the union of solutions of Equation \eqref{radsta}, see
e.g. \cite{PY1,PY2,PY3}, and these solutions are one of the ingredient
to construct sub and super-solutions to (\ref{parab}), see
e.g. \cite{W,GNW1}.

 We briefly review some known results concerning (\ref{laplace}) and
(\ref{parab})--(\ref{data}).
We need to introduce some additional parameters
which play a key role in what follows: Recalling that
$2_\ast=2\frac{n-1}{n-2}$ and that $2^\ast=\frac{2n}{n-2}$, we have
\begin{equation}\label{constant}
  \begin{aligned}
    P_F<2_\ast<2^\ast<\sigma^\ast\,\,\, \textrm{ where }\,\,\,
    P_F:= 2\frac{n+1}{n}, \\
    \sigma^\ast:= \left\{ \begin{array}{ll}
        \frac{(n-2)^2-4n+8
          \sqrt{n-1}}{(n-2)(n-10)} &
        \textrm{if $n>10$}
        \\ +\infty &
        \textrm{if $n\le 10$}
      \end{array} \right.
  \end{aligned}
\end{equation}
The parameters $P_F<2_\ast<2^\ast$ are critical exponents for this
problem and their role will be specified few lines below. Here, $P_F-1$ is the so
called Fujita exponent. \smallskip

 We assume first that $f$ is of type
(\ref{eq:potential-0}), with $k_1(|x|)=k_0(r)$, $r=|x|$, and $k_0(r):=K_0
r^{\delta_0}$, where $K_0$ is a positive constant and $\delta_0>-2$.
Let us also introduce the followings
\begin{equation} \label{eq:l0}
l_0:=2 \frac{q+\delta_0}{2+\delta_0}\,\,\, \textrm{ and }\,\,\,
m_0:=\frac{2}{l_0-2}=\frac{2+\delta_0}{q-2},
\end{equation}
and $m_0=m_0(l_0)$.

In this case, whenever
$l_0>2_\ast$, we have at least a S.G.S. with
s.d. $\phi_s(x)=P_1|x|^{-m_0}$, where $P_1>0$ is a computable
constant, which is unique for $l_0\ne 2^\ast$. Also note that if
$\delta_0=0$ then $l_0=q$ and $m_0=2/(q-2)$.  Moreover, all the regular
solutions of (\ref{radsta}) have a non-degenerate zero for $l_0 \in
(2,2^\ast)$, they are G.S. with f.d. for $l_0=2^\ast$, and they are
G.S. with s.d. for $l_0>2^\ast$ (see, e.g. \cite{W}).

 Again, if $2^\ast\le
l_0<\sigma^\ast$, then all the regular solutions cross each other,
while if $l_0 \ge \sigma^\ast$ and $\alpha_2>\alpha_1$, then
$U(r,\alpha_2)>U(r,\alpha_1)$ for any $r \ge 0$, see \cite{W}. In
fact,
when the structure of positive solution of (\ref{radsta}) changes, the
asymptotic behavior of solutions of (\ref{parab})--(\ref{data})
changes too. \smallskip

 Let us recall that all the solutions of \eqref{parab}--\eqref{data} blow up in finite
time if $l_0 \le P_F$, so the null solution is unstable in any
reasonable sense (see \cite{Fu,KST}). If $l_0>P_F$ the null solution is
stable with the suitable weighted $L^{\infty}$-norm, but still
``large" solutions blow up in finite time.

 There are several papers
devoted to explore the threshold between the basin of attraction of
the null solution and the set of initial data which blow up in finite
time (see, e.g. \cite{W,GNW1,GNW2} andalso \cite{PY1,PY2,PY3}).  It
seems that radial G.S. of (\ref{laplace}) play a key role in defining
such a border.  In particular Gui et al. in \cite{GNW1} (see also
\cite{W}) proved the following:
\begin{thm}\label{base}\cite{GNW1}
  Assume $f(u,r)=K_0 r^{\delta_0} u^{q-1}$ and $2^\ast \le l_0 <
  \sigma^\ast$. Then
  \begin{itemize}
  \item[\textbf{(1)}:] If $\phi(x) \lneqq U(|x|,\alpha)$ for some $\alpha>0$,
    then $\| u(x, t ; \phi) \|_{\infty} \to 0$ as $t \to
    +\infty$. \\[-0.3 cm]
  \item[\textbf{(2)}:] If $\phi(x) \gneqq U(|x|,\alpha)$ for some $\alpha>0$,
    then $u(x, t ; \phi) $ must blow up in finite time, i.e.  there is
    $T_{\phi} \in (0,\infty)$ such that $\lim_{t \to T_{\phi}}
    \|u(x,t;\phi)\|_{\infty}=+\infty$.
  \end{itemize}
\end{thm}
This result was extended in \cite{DLL} to potentials $f$ of the form
\eqref{eq:potential-0} where 
\begin{equation}\label{K-k}
k_1(|x|) =k_0(r):=r^{\delta_0}K(r)
\end{equation}
with $K(r)$ varying monotonically between two positive
  constants.  Then, in \cite{YZ} it was extended to $f$ of the form
(\ref{eq:potential-1}) where $k_i(r)=r^{\delta_i}k_i$ and $k_i>0$ is a
constant.  An interesting related topic 
 is the rate of decay of fading solutions and of blow up
(see, e.g. \cite{W,FWY}).

It is worth mentioning that in \cite{W,GNW1} Wang et al.
proved that for $f(u)=u^{q-1}$ the situation  is very different if $q
\ge \sigma^\ast$, i.e. G.S. are stable and weakly asymptotically
stable with the suitable weighted $L^{\infty}$-norm.   This result
still holds true also for $f$ as in (\ref{eq:potential-0}) if
$k_0(r)=r^{\delta}K(r)$, when $K$ is decreasing, uniformly positive
and bounded, and $l_0\ge \sigma^\ast$, where $l_0$ is defined in
\eqref{eq:l0}. The same result holds  for
\eqref{eq:potential-1} when $k_i=1$, for $i=1,2$ and $q_2 > q_1\ge
\sigma^\ast$ (see \cite{YZ}).  The extension of this stability results
to the potentials considered in this paper will be object of future
investigations.   \smallskip

 A first contribution of the present paper is the extension of
Theorem~\ref{base} to a number of non-linearities including
\eqref{eq:potential-0} and \eqref{eq:potential-1}. In fact, we propose a
unifying approach that allows us to consider a wider class of
non-linearities including, e.g. \eqref{eq:potential-2.5} among others.

 As we have already seen, the sub- and supercriticality of
\eqref{parab} in the non-homogenous case, e.g. if the potential $f$ is as in
\eqref{eq:potential-0}, depends on the interplay between the exponent
$q=q_1$ and the asymptotic behaviour of $k=k_1$.  The same happens for the
asymptotic behavior of positive solutions to (\ref{radsta}).
Therefore we define the following parameter, useful to combine the two
effects:
\begin{equation}\label{constML}
  l=l(q,\delta):=2\frac{q+\delta}{2+\delta} \,\,\, \textrm{ and }\,\,\,
  m(l):=\frac{2}{l-2} .
\end{equation}
If $f(u,r)=r^{\delta}u^{q-1}$ as in the Wang case, we have a
subcritical behavior for $l<2^*$ and supercritical behavior for
$l>2^*$; the same happens with the other critical parameters defined in
(\ref{constant}).

 We stress that singular and slow decay solutions
$U(r)$ of Equation \eqref{radsta} behave as $\sim P_1 r^{-m(l)}$ as $r \to 0$
and as $r \to +\infty$ respectively (where $P_1$ is a computable
constant).

Using different values of $l$ we can allow two different
behaviors for singular and slow decay solutions, namely:
Denote by $l_u$ and $m(l_u)$ the parameters ruling the asymptotic
  behavior of singular solutions $U(r)$, i.e. $U(r) \sim r^{-m(l_u)}$
  as $r \to 0$; similarly, $l_s$ and $m(l_s)$ are the parameters ruling the
  asymptotic behavior of slow decay solutions $V(r)$, i.e. $V(r) \sim
  r^{-m(l_s)}$ as $r \to +\infty$.

 This will allow
us also to consider Matukuma potentials (see below, and see also
\cite{Yanagida-1996}):  Thus, e.g. in the case
\eqref{eq:potential-0} with $k_1=k$ as in \eqref{asintotico}, we have
\begin{equation}\label{esempioM}
  f \textrm{ as in } \eqref{eq:potential-0}, \textrm{ then }
  \left\{ \begin{array}{l} l_u= 2\frac{q+\delta}{2+\delta}\\[0.4 em]
      m(l_u)= \frac{2+\delta}{q-2}
    \end{array} \right. \quad \textrm{ and }\quad
   \left\{ \begin{array}{l} l_s= 2\frac{q+\eta}{2+\eta} \\[0.4 em]
      m(l_s)= \frac{2+\eta}{q-2}
    \end{array} \right.
\end{equation}
 Analogously, in the cases \eqref{eq:potential-1} and \eqref{eq:potential-2.5}
 we have, respectively, that
\begin{equation}\label{esempioMbis}
 \begin{array}{lll}
    f \textrm{ as in } \eqref{eq:potential-1}, \textrm{ then}  \!\!& \!\!  l_u\!=\! \max\big\{
    2\frac{q_i+\delta_i}{2+\delta_i} \mid  i=1,2 \big\} , \!\!
    &\!  l_s\! =\! \min\big\{2\frac{q_i+\eta_i}{2+\eta_i} \mid
    i=1,2 \big\}\!\!\! \\[0.5 em]
    f \textrm{ as in }  \eqref{eq:potential-2.5}, \textrm{ then} \!\! &\!\!   l_u \! =
    2\frac{q_2+\delta}{2+\delta}, \!\! &\!  l_s \! =  2\frac{q_1+\eta}{2+\eta}\!\!\!\\
  \end{array}\!\!\!\!
\end{equation}
and according to (\ref{constML}) we also obtain $m(l_u)=\frac{2}{l_u-2}$ and
$m(l_s)=\frac{2}{l_s-2}$. \smallskip

 Let us state the following sub and super-criticality conditions
  related to $k_i(r)$, $i=1,2$, that
replace the fact that, for $k_0(r)=r^{\delta_0} K(r)$ in \eqref{K-k}, $K(r)$ is monotone:
\begin{itemize}
\item[$\boldsymbol{H+}$:] $\int_0^r s^{\frac{n-2}{2}q_i}
  \frac{d}{ds}[k_i(s)s^{\frac{n-2}{2}(2^\ast-q_i)}] ds \ge 0$ for any
  $r>0$ and any $i$, strictly for some $i$ and $r>0$.\\[-0.3 cm]
\item[$\boldsymbol{H-}:$] $\int_0^r s^{\frac{n-2}{2}q_i}
  \frac{d}{ds}[k_i(s)s^{\frac{n-2}{2}(2^\ast-q_i)}] ds \le 0$ for any
  $r>0$ and any $i$, strictly for some $i$ and $r>0$.
\end{itemize}
We emphasize that if $f$ is either of type \eqref{eq:potential-0},
\eqref{eq:potential-1} or \eqref{eq:potential-2.5} if $\boldsymbol{H+}$
holds then regular solutions of (\ref{radsta}) are crossing while, if
$\boldsymbol{H-}$ is verified, then they are G.S with slow decay (see
\cite{Fcamq}). \smallskip

Now we can state the following:
\begin{prop}\label{first}
  Assume that $f$ is either of the form (\ref{eq:potential-0}),
  (\ref{eq:potential-1}), (\ref{eq:potential-2.5}) and satisfies
  $\boldsymbol{H-}$.  Further, assume $ l_u \ge 2^\ast$, and $ 2^\ast
  \le l_s < \sigma^\ast$.  Then all the regular solutions
  $U(r,\alpha)$ of (\ref{radsta}) are G.S. with s.d., and there is at
  least a S.G.S. with slow decay $U(r,\infty)$. Moreover if $0<
  \alpha_1< \alpha_2 \le \infty$ for any $M>0$ there is
  $R=R(\alpha_2,\alpha_1)>M$ such that $U(R,\alpha_2)=U(R,\alpha_1)$
  and $\frac{\partial }{\partial r}U(R,\alpha_2)-\frac{\partial
  }{\partial r}U(R,\alpha_1)<0$.
\end{prop}
 This result is a direct consequence of Proposition~\ref{super} and
Remark~\ref{macchina} below.
In fact, the intersection property of G.S. is a
secondary contribution of this paper.
%
\smallskip

 In this setting we can
extend Theorem~\ref{base} as follows
\begin{thm}\label{second}
  Assume that the assumptions of Proposition~\ref{first} are verified, then 
the same conclusion as in Theorem~\ref{base} still holds true.
\end{thm}
The above result is obtained as a corollary of Theorem~\ref{unstable} below,
which is somewhat more general.

 We highlight the fact that when $f$ is of type
(\ref{eq:potential-0}), Theorem~\ref{second} generalizes the result of
\cite{DLL} to the case where $k_1(r)$ is not monotone decreasing and
may even be increasing in some cases. E.g., let   $f$ be of type
 (\ref{eq:potential-0}) with $k(r)= k_1(r)=1+r^{a}$;
assume $q \ge 2^\ast$ and $a \ge 2^\ast(q-2)-2$,  so that from
 \eqref{constML} we have $l_u=q$ and $l_s=2(q+a)/(2+a) \ge 2^*$, 
then Theorem~\ref{second} applies directly to this situation.
 \smallskip

 Notice that Theorem~\ref{second}
requires a weaker condition on $l_u$ than on $l_s$.  Hence,
Theorem~\ref{second} applies also to the case (\ref{eq:potential-0})
even for $q\ge \sigma^\ast$, with the condition that $a \in
\big(\frac{2}{\sigma^\ast-2}(q-\sigma^\ast),
\frac{n-2}{2}(q-2^\ast)\big]$, while from \cite{W} and \cite{DLL} we
know that in this case, if $k(r)$ is a constant or a decreasing
function varying between two positive values, G.S. are stable, so we
are in the opposite situation.

Also, we emphasize that Theorem~\ref{second} extends
 \cite[Theorem~1]{DLL}
also to Matukuma type potential (see, e.g.
\cite{Yanagida-1996} for more details), which are a model in astrophysics, i.e.  to $f$
of the form (\ref{eq:potential-0}) where $q \in [2^\ast,\sigma^\ast)$
and $k(r)=1/(1+r^a)$, where $a \in (0,2-\sigma^\ast(q-2))$.

When $f$ is of type (\ref{eq:potential-1}) we extend the result in
\cite{YZ} to the case where $k_i(r)$, $i=1,2$ are $r$-dependent functions, and
we can deal with a generic family of non-linearities including
(\ref{eq:potential-1}). \smallskip

Let us go back again to the case of $f(u,r)=f(u)=u^{q-1}$: The singular solution
$\phi_s(x):=P_1 |x|^{-2/(q-2)}$ seems to play a key role in
determining the threshold between solutions converging to zero and
solutions blowing up in finite time.

 In \cite{Ni} Ni shows that if
$2_\ast<q<\sigma^*$ and $\phi(x)<\phi_s(x)$, then $u(x,t;\phi)$
converges to the null solution as $t \to +\infty$.  Let $\la_1$ denote
the first eigenvalue of the Laplace operator in the ball of radius
$r=1$; if $\liminf_{|x| \to +\infty}
\phi(x)|x|^{-2/(q-2)}>(\la_1)^{1/(q-2)}$ then $u(x,t;\phi)$ blows up
in finite time.

 Wang in \cite{W} shows that if $q \ge \sigma^\ast$ and
$\liminf_{|x| \to +\infty} \phi(x)|x|^{-2/(q-2)}>P_1$ then
$u(x,t;\phi)$ blows up in finite time. Note that this result is
optimal since, for $q \ge \sigma^\ast$, there are uncountably many
G.S. with s.d. asymptotic to $P_1|x|^{-2/(q-2)}$ as $|x| \to +\infty$.
On the other hand, in \cite[Theorem~0.2, point (ii)]{W}, Wang proved
the following.
\begin{thm}\label{wangslow}
  Consider $f(u,r)= f(u)=u^{q-1}$, where $2_\ast<q<\sigma^\ast$; then for
  any $\beta>0$ there is a radial decreasing upper solution of
  (\ref{laplace}) $\chi_{\beta}(x)$ such that $\chi_{\beta}(0)=\beta$
  and $\chi_{\beta}(x)|x|^{m(q)} \to P_1:=[m(q)(n-2-m(q))]^{1/(q-2)}$
  as $|x| \to +\infty$.  Moreover $\lim_{t \to +\infty}
  \|u(x,t;\chi_{\beta})(1+|x|^{\nu})\|_{\infty}=0$ for any
  $0<\nu<m(q)$.
\end{thm}
 In fact, the result is proved for a slightly more general potential
$f(u,r)=r^{\delta}u^{q-1}$, with a proper $\delta$.  These results seems to indicate
$r^{-2/(q-2)}$ (or more in general the decay rate of slow decay
solutions of (\ref{laplace})) as the optimal decay rate for having
solutions which are continuable for any $t\in \RR$, see the
introduction of \cite{GNW2} for a detailed discussion on such a  topic.

From now till the end of this section we consider $f$ as follows:
\begin{itemize}
\item[\textbf{(i)}:] $f$ as in (\ref{eq:potential-0}) and $k_1$ satisfies
  (\ref{asintotico}) with
  $l_u,l_s>2_\ast$;\\[-0.3 cm]
\item[\textbf{(ii)}:] $f$ as in (\ref{eq:potential-1}) and $k_i$, $i=1,2$,
  satisfy (\ref{asintotico}) with   $l_u,l_s>2_\ast$ ;\\[-0.3 cm]
\item[\textbf{(iii)}:] $f$ as in (\ref{eq:potential-2.5}) and $k$ satisfies
  (\ref{asintotico}) with $l_u,l_s>2_\ast$;
\end{itemize}
As a a consequence of Theorem~\ref{unstableslow} we generalize this
result to present setting, i.e.
\begin{thm}\label{wangslowLM}
  Assume $f$ either of the form $\mathbf{(i)}$, $\mathbf{(ii)}$, or
  $\mathbf{(iii)}$.  Assume either $\boldsymbol{H+}$ with $l_u, l_s
  \in (2_\ast,2^\ast]$ or $\boldsymbol{H-}$ with $l_u\ge 2^\ast$, $l_s
  \in [2^\ast,\sigma^\ast)$. Then we have the same conclusion as in
  Theorem \ref{wangslow} but $m(q)$ is replaced by $m(l_s)$ and $P_1$
  is replaced by the computable constant $P_1^{+\infty}$ (e.g.
  $P_1^{+\infty}=[m(l_s)(n-2-m(l_s))]^{1/(q-2)}$,  if $f$ is of type
  $\mathbf{(i)}$).
\end{thm}
We emphasize that, as far as we are aware, this result is new anytime
we consider $f(u,r)$ as in (\ref{eq:potential-0}) but $k(r)\not\equiv
r^{\delta}$, and for (\ref{eq:potential-1}) even for $k_1=k_2=1$. \smallskip

The main contribution of this paper is the following result
(consequence of the slightly more general Theorem~\ref{unstableadd1})
which goes in the opposite direction with respect to Theorem~\ref{wangslow}
(and hence to Theorem~\ref{wangslowLM}),  and shows that the situation is
really delicate.
\begin{thm}\label{main2}
  Assume $f$ either of the form $\mathbf{(i)}$, $\mathbf{(ii)}$, or
  $\mathbf{(iii)}$. Further assume that either $l_u,l_s$ are in
  $(2_\ast,2^\ast]$, and $\boldsymbol{H+}$ holds, or that $l_u,l_s$
  are in $[2^\ast,+\infty)$, and $\boldsymbol{H-}$ holds.  Then there
  are one parameter families of upper and lower radial solutions with 
  fast decay of (\ref{laplace}), denoted by $\zeta_{\tau}(x)$ and
  $\psi_{\tau}(x)$ respectively; hence $\psi_{\tau}(0)=D(\tau)=\zeta_{\tau}(0)>0$,
  $\lim_{|x| \to +\infty} |x|^{n-2} \zeta_{\tau}(x)=L_{\zeta}(\tau)$ and
  $\lim_{|x| \to +\infty} |x|^{n-2} \psi_{\tau}(x)=L_{\psi}(\tau)>0$
  where $L_{\zeta}(\tau)<L_{\psi}(\tau)$. The solution $u(x,t;\zeta)$
  blow up in finite time, while the limit $\lim_{t \to +\infty}
  \|u(x,t;\psi)(1+|x|)^{\nu}\|_{\infty}=0$ for any $0<\nu<n-2$.

  Moreover
  $\|\zeta_{\tau}(x)\|_{\infty}=\|\psi_{\tau}(x)\|_{\infty}=D(\tau)
  \to +\infty$, while $L_{\zeta}(\tau)<L_{\psi}(\tau)\to 0$ as $\tau
  \to -\infty$, while $D(\tau) \to 0$ and
  $L_{\psi}(\tau)>L_{\zeta}(\tau)\to +\infty$ as $\tau \to +\infty$.
\end{thm}
\begin{remark}\label{unstableadd3}
  For any fixed $\tau \in\RR$, $\psi_{\tau}(x) \ge \zeta_{\tau}(x)$
  when $x \in \RR^n$.  From the constructive proof it follows also
  that both $\|\zeta_{\tau}(x)(1+|x|^{\nu})\|_{\infty}$ and
  $\|\psi_{\tau}(x)(1+|x|^{\nu})\|_{\infty}$ go to $0$, as $\tau \to
  +\infty$, for any $0\le \nu < m(l_s)$, while they are uniformly
  positive for $\nu=m(l_s)$.
\end{remark}
A new aspect of Theorem \ref{main2}, besides the generality of the
potential we can deal with, is in the fact that we can find fast
decaying initial data, with $L^{\infty}$-norm arbitrarily small, which
blow up in finite time, while the critical decay indicated in
literature (also by results as Theorem \ref{wangslowLM}) for such a
phenomenon seems to be slow decay, i.e. $|x|^{-m(l_s)}$ (see
\cite{GNW2}).

We emphasize that this result is new even when $f(u, |x|)= u^{q-1}$.
Notice that, the dichotomy depicted in Theorem~\ref{main2} and in 
Corollary~\ref{main3}, just below, takes place
even for solutions slightly above or below a G.S. if we are in the
hypotheses of Theorem~\ref{second}. The novelty here is that we can
look at a much larger range of parameters and that this families of
sub and super-solution have fast decay: Thus, we can find solutions
with fast decay and $L^{\infty}$-norm small which blow up in finite
time. \smallskip

The relevance of Theorem~\ref{main2} follows from the next
corollary. This latter result is an immediate consequence of the
comparison principle.
\begin{cor}\label{main3}
  Assume that we are under the hypotheses of Theorem~\ref{main2}.
  Then for any $\ep>0$ we can find smooth function $\phi: \RR^n \to
  \RR$, such that $\|\phi\|_{\infty}<\ep$ and there is $T_{\phi}>0$
  such that the classic solution $u=u(x,t; \phi)$ of
  (\ref{parab})--(\ref{data}) satisfies $\lim_{t \to T_{\phi}}\|
  u(x,t; \phi)\|_{\infty}=+\infty$.  On the other hand we can find
  smooth function $\phi: \RR^n \to \RR$, such that
  $\|\phi\|_{\infty}>1/\ep$ and the classic solution $u=u(x,t; \phi)$
  of (\ref{parab})--(\ref{data}) is defined for any $t \ge 0$ and
  satisfies $\lim_{t \to +\infty}\| u(x,t; \phi)
  (1+|x|^{\nu})\|_{\infty}=0$ for any $0 \le \nu<n-2$.
\end{cor}
From the above corollary we see how sensitive is, with respect to the initial data,
equation (\ref{parab})--(\ref{data}): We can find ``large" initial
data $\phi$ which converge to the null solution and ``small" initial
data which blow up in finite time.  Indeed, we can also construct
initial data $\phi_1$ and $\phi_2$ such that for any $\ep>0$ small we
have $\|(\phi_1-\phi_2)(x)[1+|x|^{\nu}]\|_{\infty}< \ep$, whenever
$0<\nu<m(l_s)$, and $u(x,t; \phi_1)$ blows up in finite time, while
$u(x,t; \phi_2)$ is defined for any $t$ and has the null solution as
$\omega$-limit set. However we need to choose
$\|(\phi_1-\phi_2)(x)[1+|x|^{m(l_s)}]\|_{\infty}$ uniformly positive
and bounded.  \smallskip

\textbf{Plan of the paper.}
In Section~\ref{steadyasympt} we collect all the preliminary results
concerning regular and singular solutions of \eqref{radsta} and, in
particular, we prove new ordering
properties. Section~\ref{sec:local-existence} is devoted to prove
local existence of the solutions,
in the classical, and in the mild case giving also a new result
concerning singular solutions (which are slightly smaller than
S.G.S. of \eqref{radsta}),  using a suitable weighted $L^{\infty}$-norm. Finally,
in Section~\ref{sec:long-time-behavior}, we state and prove our main
results on stability and long time behavior of the considered
solutions.

\section{Ordering results and asymptotic estimates for the
elliptic problem.}\label{steadyasympt}
The results of this sections, which are a key point for the whole
argument, are obtained applying Fowler transformation to
(\ref{radsta}). Thus, we set
\begin{equation}\label{transf1}
\begin{aligned}
    r=e^s \, , \quad y_1(s,l)=U(r)r^{m(l)} \,,\quad  y_2(s,l)=U'(r)r^{m(l)+1}
    \\
  m(l)=\frac{2}{l-2} \, ,
    \quad  g(y_1,s;l)=f(y_1 \eu^{-m(l) s}, \eu^s) \eu^{(m(l)+2)s}
  \end{aligned}
\end{equation}
Here and in the sequel $l$ denotes a parameter which is always assumed to be larger
than $2$, so that $m(l)>0$ (see the exemplifying case in
  \eqref{constML} and also the parameters related to problem
  \eqref{si.a} below).
Using this change of variables, we pass
from \eqref{radsta} to the following system to which 
dynamical tools apply:
\begin{equation} \label{si.na} \left( \begin{array}{c}
      \dot{y}_1 \\
      \dot{y}_2 \end{array}\right) = \left( \begin{array}{cc} m(l) & 1
      \\ 0 & -[n-2-m(l)]
    \end{array} \right)
  \left( \begin{array}{c} y_1 \\ y_2  \end{array}\right) +\left(
    \begin{array}{c} 0 \\- g(y_1,s;l)\end{array}\right)
\end{equation}
In the whole section the dot indicates differentiation with respect to
$s$, and we introduce the following further notation which will be in
force in this section: We write $\boldsymbol{y}(s,\tau;
\boldsymbol{Q};\bar{l})=(y_1(s,\tau;
\boldsymbol{Q};\bar{l}),y_2(s,\tau; \boldsymbol{Q});\bar{l})$ for a
trajectory of (\ref{si.na}) where $l=\bar{l}$, evaluated at $s$ and
departing from $\boldsymbol{Q} \in \RR^2$ at $s=\tau$.

For illustrative purpose we assume first $f(u,r)= r^{\delta} u^{q-1}$,
so that we can set $l= 2\frac{q+\delta}{2+\delta}$ and system
(\ref{si.na}) reduces to the following autonomous system
\begin{equation} \label{si.a} \left( \begin{array}{c}
      \dot{y}_1 \\
      \dot{y}_2 \end{array}\right) = \left( \begin{array}{cc} m(l) & 1
      \\ 0 & -[n-2-m(l)]
    \end{array} \right)
  \left( \begin{array}{c} y_1 \\ y_2  \end{array}\right) +\left(
    \begin{array}{c} 0 \\- (y_1)^{q-1} \end{array}\right)
\end{equation}
We stress that in this case we passed from a singular non-autonomous
O.D.E.  to an autonomous system from which the singularity has been
removed.  Also note that when $\delta=0$ we can simply take $l=q$.

System (\ref{si.a}) admits three critical points for
$l>2_*=2\frac{n-1}{n-2}$: The origin $O=(0,0)$,
$\boldsymbol{P}=(P_1,P_2)$ and $-\boldsymbol{P}$, where $P_2=-m(l)P_1$
and $P_1>0$.  The origin is a saddle point and admits a
one-dimensional $C^1$ stable manifold $\overline{M}^s$ and a
one-dimensional $C^1$ unstable manifold $\overline{M}^u$, see Figure
\ref{livelli}.  The origin splits $\overline{M}^s$ (respectively
$\overline{M}^u$) in two relatively open components: We denote by
$M^s$ (resp. by $M^u$) the component which leaves the origin and
enters the semi-plane $y_1 \ge 0$. Since we are just interested on
positive solutions we will call, with a little abuse of notation,
$M^s$ and $M^u$ unstable and stable manifold.

To complete the depiction of the phase portrait in Figure
\ref{livelli}, we recall the following result (see e.g. \cite{Fjdde})
\begin{remark}\label{criticalP}
  The critical point $\boldsymbol{P}$ of (\ref{si.a}) is an unstable
  node for $2_*<l \le \sigma_*$, an unstable focus if $\sigma_* < l<
  2^*$, a center if $l=2^*$, a stable focus if $2^*<l<\sigma^*$ and a
  stable node if $l \ge \sigma^*$, where $2_*,2^*,\sigma^*$ are as in
  (\ref{constant}) and
\begin{equation} \label{eq:sigma-basso}
\sigma_* :=  2\frac{n-2+2 \sqrt{n-1}}{n+2 \sqrt{n-1}-4}.
\end{equation}
\end{remark}

From some asymptotic estimate we deduce the following useful result
(see, e.g. \cite{FArch,Fcamq} for the proof in the $p$-Laplace
context).

\begin{remark}\label{corrispondenze2}
  Regular solutions $u(r)$ of Equation (\ref{radsta}) correspond to
  trajectories $\boldsymbol{Y}(s;l)$ of system (\ref{si.na}) departing
  from points in $M^u$ and viceversa.  Positive solutions with fast
  decay $u(r)$ of (\ref{radsta}), correspond to trajectories
  $\boldsymbol{Y}(s;l)$ of system (\ref{si.na}) departing from points
  in $M^s$ and viceversa.
\end{remark} Using the Pohozaev identity introduced in \cite{Po} and
adapted to this context in \cite{FArch}, we can draw a picture of the
phase portrait of (\ref{si.a}), see Figure \ref{livelli}, and deduce
information on positive solutions of (\ref{radsta}); we postpone a
sketch of the proof to the next subsection, where the general
non-autonomous case is considered (anyway see \cite{FArch} or
\cite{Fcamq} for a detailed proof in the more general $p$-Laplace
context).  Then it is easy to classify positive solutions: In the
supercritical case ($l>2^*$) all the regular solutions are G.S. with
slow decay, there is a unique S.G.S. with slow decay; in the critical
case ($l=2^*$) all regular solutions are G.S. with fast decay and
there are uncountably many S.G.S. with slow decay; in the subcritical
case ($2<l<2^*$) all the regular solutions are crossing, there are
uncountably many S.G.S. with fast decay and a unique S.G.S. with slow
decay.

Since (\ref{si.a}) is autonomous we also get the following useful
consequence.
\begin{remark}\label{tinvariant}
  Fix $\boldsymbol{U} \in M^u$ and $\boldsymbol{S} \in M^s$. Consider
  the trajectories $\boldsymbol{y}(s,\tau;\boldsymbol{U})$,
  $\boldsymbol{y}(s,\tau;\boldsymbol{S})$ of (\ref{si.na}) and the
  corresponding regular solution $U(r,D)$ and fast decay solution
  $V(r,L)$ of (\ref{radsta}).  Then
  \begin{equation*}
    \begin{array}{cc}
      D=D(\tau)=D(0) \eu^{-m\tau}  & L=L(\tau)=L(0) \eu^{(n-m) \tau}  \\
      U(r,D)=D U(rD^{1/m},1) & V(r,L)=L V(rL^{1/m},1)
    \end{array}
  \end{equation*}
\end{remark}
\begin{proof}
  Since $y_1(s+\tau,\tau,\boldsymbol{Q})=y_1(s,0,\boldsymbol{Q})$ we
  get $U(r\eu^{\tau},D(\tau))\eu^{m(l) \tau}=U(r,D(0))$ hence letting
  $r \to 0$ we find $D(\tau)=D(0) \eu^{-m(l)\tau}$, and this concludes
  the proof concerning $U$.  Similarly we find
  $V(\eu^{s+\tau},L(\tau))\eu^{n(s+\tau)}\eu^{(m-n)\tau}=V(\eu^{s},L(0))\eu^{n
    s}$, hence, letting $s \to +\infty$ we get
  $L(\tau)=L(0)\eu^{(n-m)\tau} $. Then again from
  $y_1(s+\tau,\tau,\boldsymbol{Q})=y_1(s,0,\boldsymbol{Q})$ we get
  $V(r,L)=L V(rL^{1/m},1)$:
  this concludes the proof. 
\end{proof}
\begin{figure}
  \centering
  \includegraphics*[totalheight=15cm]{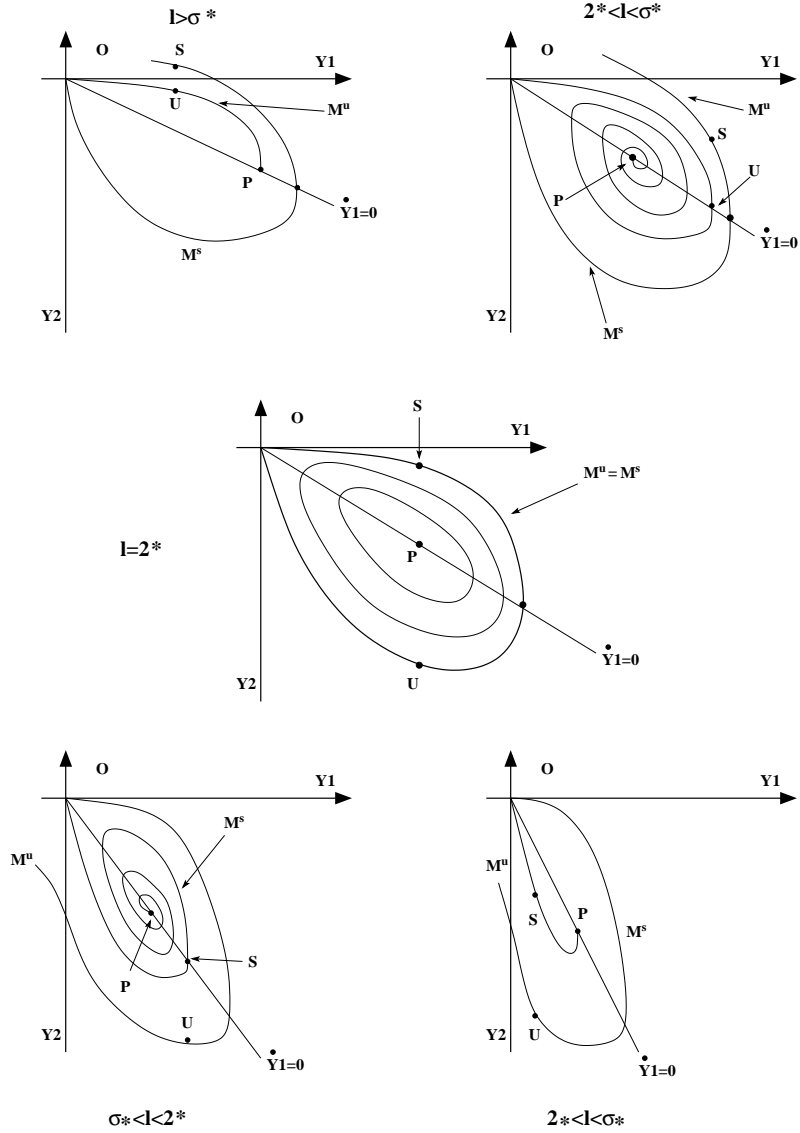}
  \caption{Sketches of the phase portrait of (\ref{si.na}), for $q>2$
    fixed.}\label{livelli}
\end{figure}
 We stress that all the previous  arguments concerning the autonomous
Equation \eqref{si.na} still hold true for any autonomous
super-linear system \eqref{si.na}, more precisely whenever
$g(y_1,s;l)\equiv g(y_1;l)$ and $g(y_1;l)$ has the following property,
denoted by $\boldsymbol{G0}$ (see \cite{Fcamq} for a proof in the
general $p$-Laplace context). We have the following
\begin{itemize}
\item[$\boldsymbol{G0}$:] $\frac{\partial g}{\partial y_1}(y_1;l)$ is a
  strictly increasing function for $y_1>0$ and
  $$ \lim_{y_1 \to 0}\frac{g(y_1;l)}{y_1} =0 \, , \qquad
  \lim_{y_1 \to +\infty} \frac{g(y_1;l)}{y_1}=+\infty \,.$$
\end{itemize}
In particular Remarks \ref{criticalP}, \ref{corrispondenze2},
\ref{tinvariant} continue to hold (see \cite{Fcamq}).

We emphasize that $\boldsymbol{G0}$ implies that
$\frac{g(y_1;l)}{y_1}$ is strictly increasing for $y_1>0$; then it
follows easily that $g(y_1;l)$ is strictly increasing too. \smallskip

To draw correctly the analogous of Figure \ref{livelli} for the
present case, we need to use the Pohozaev identity introduced in
\cite{Po} (see also \cite{Fcamq, Fow} for more details).  Let us
introduce the Pohozaev function
\begin{equation*}
  P(u, u'; r) :=  \frac{n - 2}{2} r^{n-1} uu' + \frac{r^n|u'|^2}{2} +
  F(u, r)r^n,
\end{equation*}
where $F(u, r)=\int_0^uf(a, r)da $.  Now, consider the non-autonomous
system (\ref{si.na}) and denote by $G(y_1,s;l)=\int_0^{y_1}g(a,s;l)
da$. In this dynamical setting the transposition of $P(u, u'; r)$ is
given by
\begin{equation*}
  H(y_1,y_2,s;l)= \frac{n-2}{2}y_1 y_2 +\frac{|y_2|^2}{2}+G(y_1,s;2^*),
\end{equation*}
then, if $\boldsymbol{y}(s;2^*)= (y_1(s,2^*),y_2(s,2^*))$ solves
(\ref{si.na}) with $l=2^*$ we have the following
\begin{equation}\label{derivo}
  \frac{d H}{d s} (y_1(s,2^*),y_2(s,2^*),s;2^*)=
  \frac{\partial G}{\partial s} (y_1(s,2^*),s;2^*)
\end{equation}
Moreover, if $\boldsymbol{y}(s;2^*)$ and $\boldsymbol{y}(s;l)$ are
trajectories of \eqref{si.na} corresponding to the same solution
$U(r)$ of \eqref{radsta}, we get
\begin{equation}\label{relH}
  H(\boldsymbol{y}(s,2^*),s;2^*)=\eu^{A(l)s}
  H(\boldsymbol{y}(s,l),s;l) \, .
\end{equation}
where $A(l)=n-2-2m(l)$.  We stress that (\ref{derivo}) and
(\ref{relH}) hold for the general non-autonomous system (\ref{si.na}).

Let us fix $\tau \in \RR$ and $l>2_*$ and denote by
\begin{equation}\label{Kb}
\begin{aligned}
    K(b) & :=\{(y_1,y_2) \mid H(y_1,y_2,\tau;l)=b \} \\
    K_+(b) & :=\{(y_1,y_2) \mid H(y_1,y_2,\tau;l)=b \, , \; y_1>0 \}
  \end{aligned}
\end{equation}
Then, there is $b^*(\tau,l)<0$ such that the level sets $K(b)$ of the
function $H$, is empty for $b< b^*(\tau,l)$, they are two closed
bounded curves contained in $y_2<0<y_1$ and in $y_1<0<y_2$ for
$b^*<b<0$ (the graph of the former gives $K_+(b)$), $K(b)$ is a
$8$-shaped curve having the origin as center for $b=0$, and it is a
closed bounded
curve surrounding the origin for $b>0$. \smallskip

From (\ref{derivo}) we see that $H(\boldsymbol{y}(s,2^*),s;2^*)$ is
increasing in $s$ (respectively decreasing) along the trajectories
$\boldsymbol{y}(s,2^*)$ of (\ref{si.na}) whenever $G(y_1,s; 2^*)$ is
increasing in $s$ (resp. decreasing in $s$).  Moreover from
(\ref{relH}) we see that $H(\boldsymbol{y}(s,2^*),s;2^*)$ and
$H(\boldsymbol{y}(s,l),s;l)$ have the same sign. Thus, if we consider
system (\ref{si.a}), for any $\Q \in M^u$ and $\R \in M^s$ we get
$H(\Q,s;l)<0<H(\R,s;l)$ when $l>2^*$, $H(\R,s;l)<0<H(\Q,s;l)$ when
$2<l<2^*$, and $H(\Q,s;l)=0=H(\R,s;l)$ when $l=2^*$.  Using
(\ref{derivo}) and (\ref{relH}), it can be proved that the phase
portrait of the autonomous system (\ref{si.a}) is again depicted as in
Fig. \ref{livelli}, see e.g. \cite{Fcamq,Fjdde}.

We collect here the values of several constants and parameters which
will be relevant for the whole paper.  Thus, recalling that $m(l)=\frac{2}{l-2}$,
we introduce the
followings
\begin{equation}\label{constants}
A(l)=n-2-2m(l) \, , \quad C(l)=m(l)[n-2-m(l)].
\end{equation}
Recall that $(P_1,-m(l)P_1)$ is a critical point of (\ref{si.na}) if
it is $s$-independent, so $P_1$ is the unique positive solution in $y$
of $g_l(y;l)=C(l) y$.  When $g(y,l)=y^{q-1}$ then
$P_1=(C(l))^{1/(q-2)}$.  Let $n>2$ we denote by $\sigma_*<\sigma^*$
the real solutions of the equation in $l$ given by
\begin{equation}\label{defsigma}
  A(l)^2-4[C(l)+\frac{\partial g}{\partial y}(P_1,l)]=0 \, .
\end{equation}
which reduces to $A(l)^2-4(q-2)C(l)=0$ for $g(y)=y^{q-1}$.  In this
case the value of $\sigma^*$ coincide with the one given in
(\ref{constant}).

\subsection{The stationary problem: the spatial dependent case.}
\label{subsec-stationary-prob}
Now we turn to consider (\ref{si.na}) in the $s$-dependent case. The
first step is to extend invariant manifold theory to the
non-autonomous setting; there are several ways to achieve the result:
using skew-product semi-flow (see, e.g.  \cite{JPY}), or through
Wazewski's principle, see e.g. \cite{Fcamq}.  Here, we follow a simpler
construction which is less general but preserve more properties (in
particular the ordering results Propositions~\ref{disorder+},
\ref{disorder-}),
used e.g. in \cite{Fjdde,Fdie}.  So
we introduce an extra variable, either $z(s)=\eu^{\varpi s}$ or
$\zeta(s)= \eu^{-\varpi s}$, in order to deal with a $3$-dimensional
autonomous system. We use $z$ and $\zeta$ in order to investigate the
behavior respectively as $s \to -\infty$ (i.e. $r \to 0$), and as $s
\to +\infty$ (i.e. $r \to +\infty$).

We collect here below the assumptions used in the main results:
\begin{itemize}
\item[$\boldsymbol{Gu}$:] There is $l_u > 2_*$ such that for any
  $y_1>0$ the function $ g(y_1,s;l_u)$ converges to a $s$-independent
  locally Lipschitz function $g(y_1,-\infty; l_u)\not\equiv 0$ as $s
  \to -\infty$, uniformly on compact intervals.  The function
  $g(y_1,-\infty; l_u)$ satisfies $\boldsymbol{G0}$.  Moreover there
  is $\varpi>0$ such that $\lim_{s \to -\infty} \eu^{-\varpi
    s}\frac{\partial}{\partial s} g(y_1,s;l_u)=0$.  Furthermore if
  $l_u=2^*$, we also assume that there is $M>0$ such that
  $g(y_1,s;2^*)$ is monotone in $s$ for for any $y_1>0$ and any
  $s<-M$. \\[-0.3 cm]
\item[$\boldsymbol{Gs}$:] There is $l_s >2_*$ such that for any $y_1>0$
  the function $ g(y_1,s;l_s)$ converges to a $s$-independent locally
  Lipschitz function $g(y_1; l_s)\not\equiv 0$ as $s \to +\infty$,
  uniformly on compact intervals. The function $g^{+\infty}(y_1;l_s)$
  satisfies $\boldsymbol{G0}$.  Moreover there is $\varpi>0$ such that
  $\lim_{s \to +\infty} \eu^{+\varpi s}\frac{\partial}{\partial s}
  g(y_1,s;l_s)=0$.  Furthermore if $l_s=2^*$, we also assume that
  there is $M>0$ such that $g(y_1,s;2^*)$ is monotone in $s$ for for
  any $y_1>0$ and any $s>M$. \\[-0.3 cm]
\item[$\boldsymbol{A^-}$:] The function
  $G(y_1,s;2^*):=\int_0^{y_1}g(a,s;2^*) da$ is decreasing in $s$ for
  any $y_1 >0$ strictly for some $s$. \\[-0.3 cm]
\item[$\boldsymbol{A^+}$:] $G(y_1,s;2^*)$ is increasing in $s$ for any
  $y_1 >0$ strictly for some $s$.
\end{itemize}
 Hypotheses $\boldsymbol{Gu}$, $\boldsymbol{Gs}$ are used to construct
unstable and stable manifolds for the Equation \eqref{si.na} when it
depends on $s$, while $\boldsymbol{A^-}$ and $\boldsymbol{A^+}$ mean
that the system is respectively supercritical and subcritical with
respect to $q=2^*$, and are used to understand the position of these
manifolds.
\begin{remark}\label{spiegoluls}
  Observe that if $f$ is as in (\ref{eq:potential-0}),
  (\ref{eq:potential-1}), (\ref{eq:potential-2.5}) and
  (\ref{asintotico}) hold then $\boldsymbol{Gu}$ and $\boldsymbol{Gs}$
  hold with $l_u$ and $l_s$ defined as in (\ref{esempioM}) and
  (\ref{esempioMbis}).
\end{remark}

Assume
$\boldsymbol{Gu}$. We introduce the following $3$-dimensional
autonomous system, obtained from (\ref{si.na}) by adding the extra
variable $z=\textrm{e}^{\varpi t}$:
\begin{equation}\label{si.naa}
  \begin{pmatrix}
    \dot{y}_{1} \\
    \dot{y}_{2} \\
    \dot{z}
  \end{pmatrix}
  = \begin{pmatrix}
    m(l_u) &
    1 &0
    \\ 0 & -[n-2-m(l_u)] & 0 \\
    0 & 0 & \varpi
  \end{pmatrix}
  \begin{pmatrix}
    y_1 \\ y_2 \\ z
  \end{pmatrix}+
  \begin{pmatrix}
    0 \\-
    g(y_1,\frac{\ln(z)}{\varpi};l_u)\\
    0
  \end{pmatrix}
\end{equation}
Similarly if $\boldsymbol{Gs}$ is satisfied we set $l=l_s$ and
$\zeta(t)= \eu^{-\varpi t}$ and we consider
\begin{equation}\label{si.naas}
  \begin{pmatrix}
    \dot{y}_{1} \\
    \dot{y}_{2} \\
    \dot{\zeta}
  \end{pmatrix} =
  \begin{pmatrix}
    m(l_s) &
    1 & 0 \\
    0 & -[n-2-m(l_s)] & 0 \\
    0 & 0 & -\varpi
  \end{pmatrix}
  \begin{pmatrix}
    y_1 \\ y_2 \\ \zeta
  \end{pmatrix}
  + \begin{pmatrix}
    0  \\-
    g(y_1,-\frac{\ln(\zeta)}{\varpi};l_s)\\ 0
  \end{pmatrix}
\end{equation}
The technical assumption at the end of $\boldsymbol{Gu}$ (and
$\boldsymbol{Gs}$) is needed in order to ensure that the system is
smooth for $z=0$ and $\zeta=0$ too.  Consider (\ref{si.naa})
(respectively (\ref{si.naas})) each trajectory that may be continued
for any $s \le 0$ (resp. for any $s \ge 0$) is such that its
$\alpha$-limit set is contained in the $z=0$ plane (resp. its
$\omega$-limit set is contained in the $\zeta=0$ plane); moreover such
a plane is invariant and the dynamics reduced to $z=0$
(resp. $\zeta=0$) coincide with the one of the autonomous system
(\ref{si.na}) where $g(y_1,s; l_u)\equiv g(y_1,-\infty; l_u)$ (resp.
$g(y_1,s;l_s)\equiv g(y_1,+\infty; l_s)$).

 Observe that the origin of
(\ref{si.naa}) admits a $2$-dimensional unstable manifold
$\mathbf{W^u}(l_u)$ which is transversal to $z=0$ (and a one
dimensional stable manifold $M^s$ contained in $z=0$), while the
origin of (\ref{si.naas}) admits a $2$-dimensional stable manifold
$\mathbf{W^s}(l_s)$ which is transversal to the plane $\zeta=0$ (and a
one dimensional unstable manifold $M^u$ contained in $\zeta=0$).
Following \cite{Fdie}, see also \cite{JPY,Fcamq} we see that, for any
$\tau \in \RR$,
\begin{equation*}
  \begin{array}{cc}
    W^u(\tau;l_u)=\boldsymbol{W^u}(l_u) \cap
    \{ z= \eu^{\varpi \tau} \}\,, \; & \;
    W^s(\tau;l_s)=\boldsymbol{W^s}(l_s) \cap
    \{ \zeta= \eu^{-\varpi \tau} \}\\
    W^u(-\infty;l_u)=\boldsymbol{W^u}(l_u) \cap
    \{ z= 0 \}\,, \; & \; W^s(+\infty;l_s)=\boldsymbol{W^s}(l_s) \cap \{ \zeta= 0 \}
  \end{array}
\end{equation*}
\noindent are one-dimensional manifolds.  Moreover they inherit the
same smoothness as (\ref{si.naa}) and (\ref{si.naas}). I.e., let $K$
be a segment which intersects $W^u(\tau_0;l_u)$ (respectively
$W^s(\tau_0;l_s)$) transversally in a point $\Q(\tau_0)$ for $\tau_0
\in[-\infty,+\infty)$ (respectively for $\tau_0
\in(-\infty,+\infty]$), then there is a neighborhood $I$ of $\tau_0$
such that $W^u(\tau;l_u)$ (respectively $W^s(\tau;l_s)$) intersects
$K$ in a point $\Q(\tau)$ for any $\tau \in I$, and $\Q(\tau)$ is as
smooth as (\ref{si.naa}) (resp. as (\ref{si.naas})).  Since we need to
compare $W^u(\tau;l_u)$ and $W^s(\tau;l_s)$ we introduce the
manifolds:
\begin{equation}\label{cambioL}
  \begin{split}
    &W^u(\tau;l_s):= \big\{ \R=\Q
    \textrm{exp}\big\{-\big(m(l_u)-m(l_s)\big)\tau
    \big\} \in \RR^2 \mid \Q \in W^u(\tau;l_u) \big\} \\
    &W^s(\tau;l_u):= \big\{ \Q=\R
    \textrm{exp}\big\{\big(m(l_u)-m(l_s)\big)\tau \big\} \in \RR^2
    \mid \R \in W^s(\tau;l_s) \big\}
  \end{split}
\end{equation}
As in the $s$-independent case, we see that \emph{regular solutions
  correspond to trajectories in $W^u$ while fast decay solutions
  correspond to trajectories in $W^s$}, see \cite{Fdie,Fcamq}.  More
precisely, from Lemma 3.5 in \cite{Fcamq} we get the following.
\begin{lemma}\label{corrispondenze}
  Consider the trajectory $\boldsymbol{y}(s,\tau,\Q;l_u)$ of
  (\ref{si.na}) with $l=l_u$, the corresponding trajectory
  $\boldsymbol{y}(t,\tau,\R; l_s)$ of (\ref{si.na}) with $l=l_s$ and
  let $u(r)$ be the corresponding solution of (\ref{radsta}). Then
  $\R=\Q \textrm{exp}[ (m(l_s)-m(l_u)) \tau]$.  Assume
  $\boldsymbol{Gu}$; then $u(r)$ is a regular solution if and only if
  $\Q \in W^u(\tau; l_u)$ or equivalently $\R \in W^u(\tau; l_s)$.
  Assume $\boldsymbol{Gs}$; then $u(r)$ is a fast decay solution if
  and only if $\R \in W^s(\tau; l_s)$ or equivalently $\Q \in
  W^s(\tau; l_u)$.\\
  Moreover if $\Q \in W^s(\tau;l_u)$ and $l_u>2_*$ then $\lim_{s \to
    +\infty}\boldsymbol{y}(s,\tau,\Q;l_u)=(0,0)$, and if $\R \in
  W^u(\tau;l_s)$ and $l_s>2$ then $\lim_{s \to
    -\infty}\boldsymbol{y}(s,\tau,\R;l_s)=(0,0)$.
\end{lemma}
For the reader's convenience we now report a result proved in
\cite{Fcamq} which explains further the relationship between
(\ref{radsta}) and (\ref{si.na}): We recall that, close to the origin,
$W^u(\tau;l_u)$ is locally a graph on the $y_1$ axis, while
$W^s(\tau;l_s)$ is locally a graph on its tangent space, i.e.  the
line $y_2=-(n-2)y_1$. So let us consider a ball $B(\delta)$ of radius
$\delta>0$ centered in the origin.  Follow $W^u(\tau;l_u)$
(respectively $W^s(\tau;l_s)$) from the origin towards $y_1>0$. If
$\delta>0$ is small enough, we can choose a segment $L \subset
B(\delta)$, parallel to the $y_2$ axis such that $W^u(\tau;l_u)$
(respectively $W^s(\tau;l_s)$) intersects $L$ transversally a first
time exactly in a point, say $\boldsymbol{Q^u}(\tau)$
(resp. $\boldsymbol{Q^s}(\tau)$).  We know that this point depends on
$\tau$ as smoothly as (\ref{si.na}), so it is at least $C^1$.

Moreover,
we have the following result analogous to \ref{tinvariant}, see
\cite{Fcamq,Fdie}.
\begin{remark}\label{corr}
  Assume $\boldsymbol{Gu}$. Consider
  $\boldsymbol{y}(s,\tau,\boldsymbol{Q^u}(\tau);l_u)$ and the
  corresponding regular solution $U(r,\alpha(\tau))$ of
  (\ref{radsta}). Then $\alpha(\tau) \to 0$ as $\tau \to -\infty$
  and $\alpha(\tau) \to +\infty$ as $\tau \to +\infty$.

  Similarly, assume $\boldsymbol{Gs}$. Consider
  $\boldsymbol{y}(s,\tau,\boldsymbol{Q^s}(\tau);l_s)$ and the
  corresponding fast decay solution $V(r,\beta(\tau))$ of
  (\ref{radsta}). Then $\beta(\tau) \to 0$ as $\tau \to -\infty$ and
  $\beta(\tau) \to +\infty$ as $\tau \to +\infty$.
\end{remark}
Now we turn to consider singular and slow decay solutions of (\ref{radsta}).\\
We observe that if $l_u >2_*$ then (\ref{si.naa}) has a critical point
in $y_1>0$, say $(\boldsymbol{P^{-\infty}},0)$, where
$\boldsymbol{P^{-\infty}}=(P_1^{-\infty},-m(l_u)P_1^{-\infty})$ is the
critical point of the autonomous system (\ref{si.na}) where
$g(y_1,s;l_u) \equiv g(y_1,-\infty;l_u)$, and $P_1^{-\infty}>0$.  It
is easy to check that $(\boldsymbol{P^{-\infty}},0)$ admits an
exponentially unstable manifold transversal to $z=0$ which is
$1$-dimensional (the graph of a trajectory which will be denoted by
$\boldsymbol{y^u}(s;l_u)$) if $l_u \ge 2^*$, and $3$-dimensional if
$2_*<l_u < 2^*$.

Analogously, if $l_s>2_*$ then (\ref{si.naas}) has a critical point in
$y_1>0$, say $(\boldsymbol{P^{+\infty}},0)$, where
$\boldsymbol{P^{+\infty}}=(P_1^{+\infty},-m(l_s)P_1^{+\infty})$ is the
critical point of the autonomous system (\ref{si.na}) where
$g(y_1,s;l_s) \equiv g(y_1,+\infty;l_s)$, and $P_1^{+\infty}>0$.
$(\boldsymbol{P^{+\infty}},0)$ admits an exponentially stable manifold
transversal to $\zeta=0$ which is $1$-dimensional (the graph of a
trajectory which will be denoted by $\boldsymbol{y^s}(s;l_s)$) if
$2_*<l_s \le 2^*$ and $3$-dimensional if $l_s>2^*$.  In the whole
paper we denote by $U(r,\infty)$ the solution of (\ref{radsta})
corresponding to $\boldsymbol{y^u}(s;l_u)$ and by
$\boldsymbol{y^u}(s;l_s)$ the corresponding trajectory of
(\ref{si.na}) with $l=l_s$; similarly we denote by $V(r,\infty)$ the
slow decay solution corresponding to $\boldsymbol{y^s}(s;l_s)$ and by
$\boldsymbol{y^s}(s;l_u)$ the corresponding trajectory of
(\ref{si.na}) with $l=l_u$.

Note that if $2_*<l_s<2^*<l_u$ then the manifolds $W^u(-\infty;l_u)$
and $W^s(+\infty;l_s)$ are paths connecting the origin respectively
with $\boldsymbol{P^{-\infty}}$ and $\boldsymbol{P^{+\infty}}$, and
contained in $y_2<0<y_1$ (we emphasize that this is not the case when
$2_*<l_u \le 2^* \le l_s$), see Figure \ref{livelli}.  Using a
connection argument we get the following.
\begin{remark}\label{connection}
  Assume $\boldsymbol{Gu}$, $\boldsymbol{Gs}$ with $2_*<l_s<2^*<l_u$,
  then $W^u(\tau;l_u)$ and $W^u(\tau;l_s)$ are paths connecting the
  origin respectively with $\boldsymbol{y^{u}}(\tau;l_u)$ and
  $\boldsymbol{y^{u}}(\tau;l_s)$ for any $\tau \in [-\infty,+\infty)$;
  similarly $W^s(\tau;l_u)$ and $W^s(\tau;l_s)$ are paths connecting
  the origin respectively with $\boldsymbol{y^{s}}(\tau;l_u)$ and
  $\boldsymbol{y^{s}}(\tau;l_s)$ for any $\tau \in (-\infty,+\infty]$
\end{remark}
\begin{remark}\label{disorder+}
  Assume $\boldsymbol{Gu}$ with $l_u>2_*$; then there is at least one
  singular solution $U(r,\infty)$ of (\ref{radsta}).  Moreover
  $U(r,\infty)r^{m(l_u)}$ converges to $P_1^{-\infty}$ as $r \to 0$.
  Furthermore $U(r,\infty)$ is the unique singular solution if
  $l_u>2^*$.
\end{remark}
A specular argument gives us a similar condition for slow decay
solutions.
\begin{remark}\label{disorder-}
  Assume $\boldsymbol{Gs}$ with $l_s>2_*$; then there is at least one
  slow decay solution $V(r,\infty)$ of (\ref{radsta}): moreover
  $V(r,\infty)r^{m(l_s)}$ converges to $P_1^{+\infty}$ as $r \to
  +\infty$.  Such a solution is unique if $2_*<l_s<2^*$.
\end{remark}
Now we give a further result concerning separation properties which
will be useful to construct sub and super-solutions for
(\ref{radsta}).
\begin{remark}\label{macchina}
  Assume $\boldsymbol{Gs}$ with $l_s \in [2^*, \sigma^*)$ and consider
  two slow decay solutions $\bar{U}(r)$ and $\tilde{U}(r)$ of
  (\ref{radsta}). Then $\bar{U}(r)-\tilde{U}(r)$ changes sign
  infinitely many times as $r \to +\infty$.  Analogously, assume
  $\boldsymbol{Gu}$ with $l_u \in (\sigma_*,2^*]$ and consider two
  singular solutions $\bar{V}(r)$ and $\tilde{V}(r)$ of
  (\ref{radsta});  with $\sigma^*$ and $\sigma_*$
    in \eqref{constant} and \eqref{eq:sigma-basso}, respectively.
Then $\bar{V}(r)-\tilde{V}(r)$ changes sign
  indefinitely as $r \to 0$.
\end{remark}
\begin{proof}
  Denote by $\boldsymbol{\bar{y}}(s)=\boldsymbol{\bar{y}}(s;l_s)$,
  $\boldsymbol{\tilde{y}}(s)=\boldsymbol{\tilde{y}}(s;l_s)$ the
  solutions of (\ref{si.na}) corresponding to $\bar{U}(r)$ and
  $\tilde{U}(r)$ respectively.  Now assume $l_s=2^*$ (and
  $g(y_1,s;2^*)$ monotone in $s$ for $s$ large). Then
  $H(\boldsymbol{\bar{y}}(s), s;2^*) \to \bar{b}$, and
  $H(\boldsymbol{\tilde{y}}(s), s;2^*) \to \tilde{b}$ as $s \to
  +\infty$, and $\bar{b}, \tilde{b}$ are both negative. If $\bar{b}
  \ge \tilde{b}>b^*$ then $\boldsymbol{\bar{y}}(s)$ converges to
  $K_+(\bar{b})$, and $\boldsymbol{\tilde{y}}(s)$ to $K_+(\tilde{b})$,
  see (\ref{Kb}): by construction $K_+(\tilde{b})$ lies in the
  interior of the bounded set enclosed by $K_+(\bar{b})$.  Denote by
  $\boldsymbol{A^+}$ and $\boldsymbol{A^-}$ the point of
  $K_+(\bar{b})$ respectively with largest and smallest component
  $y_1$.  When $\boldsymbol{\bar{y}}(s)$ passes close to
  $\boldsymbol{A^+}$ we have $\bar{y}_1(s)-\tilde{y}_1(s)>0$, while
  when $\boldsymbol{\bar{y}}(s)$ passes close to $\boldsymbol{A^-}$ we
  have $\bar{y}_1(s)-\tilde{y}_1(s)<0$, so the remark is proved.

  The argument works also if $\bar{b}>\tilde{b}=b^*$, so assume now
  $\bar{b}=\tilde{b}=b^*$, i.e. both $\boldsymbol{\bar{y}}(s)$ and
  $\boldsymbol{\tilde{y}}(s)$ converge to $\boldsymbol{P^{+\infty}}$.
  \\
  We denote by $h(s)=\bar{y}_1(s)-\tilde{y}_1(s)$: note that $h(s) \to
  0$ as $s \to +\infty$ and that it satisfies
  \begin{equation}\label{lin}
    \ddot{h}(s)+B h(s)+N(h(s),s)=0
  \end{equation}
  where $B=B(l_s)=\frac{\partial g^{+\infty}}{\partial
    y_1}(P_1^{+\infty})- m(l_s)[n-2-m(l_s)]= \frac{\partial
    g^{+\infty}}{\partial y_1}(P_1^{+\infty})- \frac{(n-2)^2}{4}$ and
  \begin{equation}\label{oN}
    \begin{split}
      N(h(s),s):=& g(\tilde{y}_1(s)+h(s),s)- g(\tilde{y}_1(s),s) -
      \frac{\partial g^{+\infty}}{\partial y_1}(P_1^{+\infty})h(s)= \\
      =& h(s)\int_0^1 \Big[\frac{\partial g}{\partial
        y_1}(\tilde{y}_1(s)+\sigma h(s),s) -\frac{\partial
        g^{+\infty}}{\partial y_1} (P_1^{+\infty}) \Big] d\sigma
    \end{split}
  \end{equation}
  So from (\ref{oN}), $\boldsymbol{Gs}$, and the fact that
  $|\tilde{y}_1(s)|+|h(s)| \to {P_1^{+\infty}}$ as $s \to +\infty$ we
  see that $N(h(s),s)=o(h(s))$. Therefore for any $\ep>0$ we find
  $S=S(\ep)$ such that $|N(h(s),s)| \le \ep |h(s)|$ for any $s>S$.
  Note also that from $\boldsymbol{Gs}$ we get $B>0$.  Setting
  \begin{equation}\label{teta}
    h(s)=   \rho(s) \frac{\cos(\theta(s))}{\sqrt{B}} \, , \qquad \dot{h}(s)= \rho(s) \sin(\theta(s))
  \end{equation}
  from (\ref{lin}) we get
  \begin{equation}\label{polar}
    \dot{\theta}(s)= -\sqrt{B}- \cos(\theta(s))\frac{N(\rho(s)
      \frac{\cos(\theta(s))}{\sqrt{B}},s)}{\rho(s) }< -\sqrt{B}(1-\ep)<-\frac{\sqrt{B}}{2}
  \end{equation}
  for any $s>S$ and $S$ large enough. Since $\theta(s) \to -\infty$,
  and $\rho(s) \to 0$ as $s \to +\infty$, but $\rho(s)>0$ for any $s
  \in \RR$, then $h(s)$ changes sign indefinitely, and the Remark
  follows.

  Assume now $l_s \in (2^*,\sigma^*)$: then both
  $\boldsymbol{\bar{y}}(s)$, $\boldsymbol{\tilde{y}}(s)$ converge
  exponentially to $\boldsymbol{P^{+\infty}}$, therefore
  $h(s)=\bar{y}(s)-\tilde{y}(s) \to 0$ as $s \to +\infty$.  In this
  case (\ref{lin}) is replaced by
  \begin{equation}\label{lin2}
    \ddot{h}(s)-A \dot{h}(s)+B h(s)+N(h(s),s)=0
  \end{equation}
  where $A=A(l_s)= n-2-2 m(l_s) <0$ and $B=B(l_s)=\frac{\partial
    g^{+\infty}}{\partial y_1}({P_1^{+\infty}})-
  m(l_s)[n-2-m(l_s)]>0$.  Note that $\sqrt{B}-\frac{A}{2}=2 \ep>0$ for
  $l_s \in (\sigma_*,2^*]$ and it equals to $0$ for $l_s=\sigma_*$.
  So, using again (\ref{teta}), and passing to polar coordinates we
  get
  \begin{equation}\label{polar2}
\begin{aligned}
    \dot{\theta}(s) & = -\sqrt{B}-\frac{A \sin(2 \theta(s))}{2}-
    \cos(\theta(s))\frac{N(\rho(s)
      \frac{\cos(\theta(s))}{\sqrt{B}},s)}{\rho(s) }\\
&<     -\sqrt{B}+\frac{A}{2}+\ep<-\ep
  \end{aligned}
\end{equation}
  So we find again that $\theta(s) \to -\infty$, and $\rho(s) \to 0$
  as $s \to +\infty$, thus $h(s)$ changes sign indefinitely, and the
  Remark follows.

  The case of singular solutions $\tilde{V}(r)$ and $\bar{V}(r)$ can
  be obtained from the previous repeating the argument but reversing
  the direction of $s$.
\end{proof}

Following \cite{Fcamq} we can show that if $\boldsymbol{A^-}$ holds
then (\ref{radsta}) is supercritical, while if $\boldsymbol{A^+}$
holds then (\ref{radsta}) is subcritical.  To be more precise we have
the following (see \cite[Theorems~4.2 and 4.3]{Fcamq}).
\begin{prop}\label{sub}\cite{Fcamq}
  Assume $\boldsymbol{Gu}$, $\boldsymbol{Gs}$, with $l_s, l_u \in(2_*,
  2^*]$, and $\boldsymbol{A^+}$, then all the regular solutions
  $U(r,\alpha)$ are crossing, i.e. there is $R(\alpha)$ such that
  $U(r,\alpha)>0$ for $0 \le r<R(\alpha)$ and $U(R(\alpha),\alpha)=0$.
  Furthermore $R(\alpha)$ is continuous and $R(\alpha) \to +\infty$ as
  $\alpha \to 0$, and if $l_u<2^*$ then $R(\alpha)\to 0$
  as $\alpha \to +\infty$.

  Moreover, all the fast and slow decay solutions are S.G.S.  So for
  any $\beta>0$ the fast decay solution $V(r, \beta)$ is a S.G.S. with
  fast decay; if $l_s<2^*$ there is a unique S.G.S. with slow decay,
  say $V(r,\infty)$, while if $l_s=2^*$ there are uncountably many
  S.G.S. with slow decay.
\end{prop}
\begin{proof}
  This result is borrowed from \cite[Theorem~4.2]{Fcamq} , where it
  is proved in the $p$-Laplace context in a more general framework, so
  here we just sketch the proof.  The main idea is to use the Pohozaev
  identity as done in the previous subsection: From (\ref{derivo}) we
  know that the function $H(\boldsymbol{y},s;2^*)$ is decreasing along
  the trajectories, and it is null for $\boldsymbol{y}=0$. Using also
  (\ref{relH}) we see that if $\Q \in W^u(\tau;l_u)$ and $\R \in
  W^s(\tau;l_s)$ we get $H(\Q,\tau;l_u)>0>H(\R,\tau;l_s)$.  Recalling
  which is the form of the level set $K(b)$ of $H$ (see (\ref{Kb}) and
  the discussion just after it) we deduce which is the position of
  $W^u(\tau;l_u)$ and $W^s(\tau;l_s)$ and using Lemma
  \ref{corrispondenze}, Remark \ref{connection} we conclude the proof.
\end{proof}
With a specular argument we get the following.
\begin{prop}\label{super}
  Assume $\boldsymbol{Gu}$, $\boldsymbol{Gs}$ with $l_u, l_s \ge 2^*$,
  and $\boldsymbol{A^-}$, then all the regular solutions $U(r,\alpha)$
  are G.S. with slow decay.  Moreover all the fast decay solutions
  $V(r, \beta)$ have a positive non-degenerate zero $r=R(\beta)$, i.e.
  $V(r,\beta)$ is positive for any $r>R(\beta)$ and it is null for
  $r=R(\beta)$. Furthermore $R(\beta)$ is continuous and $R(\beta)\to
  0$ as $\beta \to +\infty$, and if $l_u>2^*$ then $R(\beta) \to
  +\infty$ as $\beta \to 0$.  Further if $l_u=2^*$ there are
  uncountably many S.G.S. with slow decay, while if $l_u>2^*$ then
  there is a unique S.G.S. with slow decay say $U(r, \infty)$.
\end{prop}
Now we give a Lemma, consequence of Propositions \ref{sub} and
\ref{super}, which allows to extend picture \ref{livelli} to the
non-autonomous setting.  Assume
$\boldsymbol{A^+},\boldsymbol{Gu},\boldsymbol{Gs}$ with $2<l_u < 2^*$
and $2_*<l_s \le 2^*$.  Follow $W^u(\tau;l_u)$ from the origin towards
$\RR^2_+:=\{(y_1,y_2) \mid y_1>0 \}$: it intersects the $y_2$ positive
semi-axis in a point, say $\boldsymbol{Q^u}(\tau)$. We denote by
$\bar{W}^u(\tau;l_u)$ the branch of $W^u(\tau;l_u)$ between the origin
and $\boldsymbol{Q^u}(\tau)$, and by $\bar{E}^u(\tau)$ the bounded set
enclosed by $\bar{W}^u(\tau;l_u)$ and the $y_2$ axis.

Similarly assume $\boldsymbol{A^-},\boldsymbol{Gu},\boldsymbol{Gs}$
with $l_u \ge 2^*$ and $l_s>2^*$.  Follow $W^s(\tau;l_s)$ from the
origin towards $y_1 \ge 0$: it intersects the $y_2$ negative semi-axis
in a point, say $\boldsymbol{Q^s}(\tau)$. We denote by
$\bar{W}^s(\tau;l_s)$ the branch of $W^s(\tau;l_s)$ between the origin
and $\boldsymbol{Q^s}(\tau)$, and by $\bar{E}^s(\tau)$ the bounded set
enclosed by $\bar{W}^u(\tau;l_u)$ and the $y_2$ axis.  Using the fact
that $H(\Q,\tau;l_u)>0>H(\R,\tau;l_u)$ for any $\Q \in W^u(\tau;l_u)$,
$\R \in W^s(\tau;l_u)$ if $\boldsymbol{A^+}$ holds and $2<l_u < 2^*$
and $2_*<l_s \le 2^*$, while $H(\Q,\tau;l_s)<0<H(\R,\tau;l_s)$ for any
$\Q \in W^u(\tau;l_s)$, $\R \in W^s(\tau;l_s)$ if $\boldsymbol{A^-}$
holds and $l_u \ge 2^*$, $l_s> 2^*$ we get the following Lemma, which
is useful to construct a new family of sub and super-solutions, see
also Remark \ref{connection}.

\begin{lemma}\label{picture}
  Assume $\boldsymbol{A^+},\boldsymbol{Gu},\boldsymbol{Gs}$ with
  $2_*<l_u < 2^*$ and $2_*<l_s \le 2^*$.  Then for any $\tau \in \RR$,
  $W^s(\tau;l_u)\subset \bar{E}^u(\tau)$; assume further $l_s<2^*$,
  then $W^s(\tau;l_u)$ is a path joining the origin and
  $\boldsymbol{y^{s}}(\tau;l_u)$.\\
  Assume $\boldsymbol{A^-},\boldsymbol{Gu},\boldsymbol{Gs}$ with $l_u
  \ge 2^*$ and $l_s>2^*$.  Then for any $\tau \in \RR$,
  $W^u(\tau;l_s)\subset \bar{E}^s(\tau)$; assume further $l_u>2^*$,
  then $W^u(\tau;l_s)$ is a path joining the origin and
  $\boldsymbol{y^{u}}(\tau;l_s)$.
\end{lemma}
We emphasize that the sets $\bar{E}^u(\tau)$, $\bar{E}^s(\tau)$ have
the following property: let $\Q \in \bar{E}^u(\tau)$, $\R
\in\bar{E}^s(\tau)$, then $\boldsymbol{y}(s,\tau,\Q;l_u) \in
\bar{E}^u(t)$ for any $s \le \tau$, and $\boldsymbol{y}(s,\tau,\R;l_s)
\in \bar{E}^s(t)$ for any $s \ge \tau$.

When $l_u=l_s=2^*$ we have a slightly different situation.  Denote by
$\boldsymbol{P^*}(\tau)=(P^*_1(\tau),P^*_2(\tau))$ the critical point
of the autonomous system (\ref{si.na}) where $l=2^*$ and $g(y_1,s;
2^*)\equiv g(y_1,\tau;2^*)$. Denote by $P^*_1:= \inf \{P^*_1(\tau)
\mid \tau \in \RR \}$; in this setting we have $P^*_1>0$ and we denote
by $\bar{P}^*=P^*_1 /2$. We denote by
$\boldsymbol{\bar{P}^*}=(\bar{P}^*, -m(2^*)\bar{P}^*)$.
\begin{lemma}\label{picture2}
  Assume $\boldsymbol{Gu},\boldsymbol{Gs}$ with $l_u=l_s = 2^*$.
  Assume further $\boldsymbol{A^+}$, then for any $\tau \in \RR$ the line
  $y_1=\bar{P}^*$ intersect the manifold $W^u(\tau)$ in
  $\boldsymbol{Q^{u,+}}(\tau)=(\bar{P}^*,Q^{u,+}_2(\tau))$ and in
  $\boldsymbol{Q^{u,-}}(\tau)=(\bar{P}^*,Q^{u,-}_2(\tau))$, and it
  intersects $W^s(\tau)$ in
  $\boldsymbol{Q^{s,-}}(\tau)=(\bar{P}^*,Q^{s,-}_2(\tau))$ and
  $Q^{u,-}_2(\tau)<Q^{s,-}_2(\tau)<-m(2^*)\bar{P}^*<Q^{u,+}_2(\tau)$.
  Moreover, if $\boldsymbol{y}(s)$ corresponds to a S.G.S. with slow
  decay, there is $\Q=(Q_1,Q_2) \in W^u(\tau)$ such that
  $Q_1=y_1(\tau)$ and $Q_2>y_2(\tau)$ for any $\tau \in \RR$.

  Now, assume $\boldsymbol{A^-}$, then for any $\tau \in \RR$ the line
  $y_1=\bar{P}^*$ intersect the manifold $W^s(\tau)$ in
  $\boldsymbol{Q^{s,\pm}}(\tau)=(\bar{P}^*,Q^{s,\pm}_2(\tau))$, and
  $W^u(\tau)$ in
  $\boldsymbol{Q^{u,+}}(\tau)=(\bar{P}^*,Q^{u,+}_2(\tau))$ and
  $Q^{s,-}_2(\tau)<Q^{u,-}_2(\tau)<-m(2^*)\bar{P}^*<Q^{s,+}_2(\tau)$.
  Moreover if $\boldsymbol{y}(s)$ corresponds to a S.G.S. with slow
  decay, there is $\Q=(Q_1,Q_2) \in W^s(\tau)$ such that
  $Q_1=y_1(\tau)$ and $Q_2<y_2(\tau)$ for any $\tau \in \RR$.
\end{lemma}
\begin{proof}
  We recall that $W^u(\tau;2^*)$ and $W^s(\tau;2^*)$ depend smoothly
  on $\tau$ and that they become the graph of a homoclinic trajectory
  as $\tau \to -\infty$ and as $\tau \to +\infty$ respectively.
  Denote by
  \begin{equation*}
S(\tau):=\{(y_1,y_2) \mid H(y_1,y_2,\tau;2^*)=0 \, , \; y_1>0 \}
\end{equation*}
and by $\boldsymbol{H}^+(\tau)=(\bar{P}^*,H^+(\tau))$, and
  $\boldsymbol{H}^-(\tau)=(\bar{P}^*,H^-(\tau))$ the intersection of
  $S(\tau)$ with the line $y_1=\bar{P}^*$, where
  $H^-(\tau)<H^+(\tau)$.

  From an analysis of the phase portrait relying on Wazewski's
  principle it follows that $W^u(\tau;2^*)$ (respectively
  $W^s(\tau;2^*)$) intersects the line $y_1=\bar{P}^*$ for any $\tau
  \in\RR$, see \cite{Fosc} for a proof in the $p$-Laplace context.
  Follow $W^u(\tau;2^*)$ and $W^s(\tau; 2^*)$ from the origin towards
  $y_1>0$: we denote by $\boldsymbol{Q^{u,+}}(\tau)$ the first
  intersection of $W^u(\tau;2^*)$ (resp. of $W^s(\tau;2^*)$) with the
  line $y_1=\bar{P}^*$, and by $\boldsymbol{Q^{s,-}}(\tau)$ the first
  intersection of $W^s(\tau;2^*)$ with $y_1=\bar{P}^*$.  Using
  transversal smoothness of the manifold $W^u(\tau;2^*)$ and
  $W^s(\tau;2^*)$, see subsection 2.1, we see that we have at least a
  further intersection with such a line, respectively for $\tau<<0$
  and for $\tau>>0$.  We denote by $\boldsymbol{Q^{u,-}}(\tau_u)$ the
  second intersection of $W^u(\tau_u;2^*)$ with the line
  $y_1=\bar{P}^*$ and by $\boldsymbol{Q^{s,+}}(\tau_s)$ the second
  intersection of $W^s(\tau_s;2^*)$ with the line $y_1=\bar{P}^*$, for
  any $\tau_u\le -N$ and $\tau_s \ge N$ and $N>0$ large enough.  Set
  $\boldsymbol{Q^{u,\pm}}(\tau_u)=(\bar{P}^*,Q^{u,\pm}_2(\tau_u))$ and
  $\boldsymbol{Q^{s,\pm}}(\tau_s)=(\bar{P}^*,Q^{s,\pm}_2(\tau_s))$:
  possibly choosing a larger $N$ we can assume w.l.o.g. that
  $Q^{u,+}_2(\tau_u)>-m(2^*)\bar{P}^*>Q^{u,-}_2(\tau_u)$ and
  $Q^{s,+}_2(\tau_u)>-m(2^*)\bar{P}^*>Q^{s,-}_2(\tau_s)$.  We denote
  by $\bar{W}^u(\tau_u)$ the branch of $W^u(\tau_u;2^*)$ between the
  origin and $\boldsymbol{Q^{u,-}}(\tau_u)$ and by $\bar{W}^s(\tau_s)$
  the branch of $W^s(\tau_s;2^*)$ between the origin and
  $\boldsymbol{Q^{s,+}}(\tau_s)$.

  Assume $\boldsymbol{A^+}$; then $W^u(\tau;2^*)$ lies in the exterior
  of the bounded set enclosed by $S(\tau)$ for any $\tau$: we claim
  that $\boldsymbol{Q^{u,-}}(\tau)$ exists for any $\tau \in \RR$.  In
  fact consider the semi-line $L(\tau)=\{(P^*,y_2) \mid y_2
  <H^-(\tau)\}$; the flow of (\ref{si.na}) on $L(\tau)$ points towards
  $y_1<0$ for any $\tau \in \RR$.  Hence the trajectory
  $\boldsymbol{y}(s,-N,\boldsymbol{Q^{u,-}}(-N);2^*)$ crosses the line
  $y_1=P^*$ for $s=-N$ and then the $y_1<0$ semi-plane, and similarly
  for any $\Q \in \bar{W}^u(\tau)$ the trajectory
  $\boldsymbol{y}(s,-N,\boldsymbol{Q};2^*)$ will cross the line
  $y_1=P^*$ for a certain $s>-N$ and then the $y_1<0$ semi-plane.
  Hence, for any $\tau \ge -N$, the branch of the manifold
  $W^u(\tau;2^*)$ between the origin and
  $\boldsymbol{y}(\tau,-N,\boldsymbol{Q^{u,-}}(-N);2^*)$ will surround
  $S(\tau)$ untill it crosses a second time the line $y_1=P^*$ and the
  claim is proved, so we get picture \ref{figpicture}.
  \begin{figure}
    \centering
    \includegraphics*[totalheight=5cm]{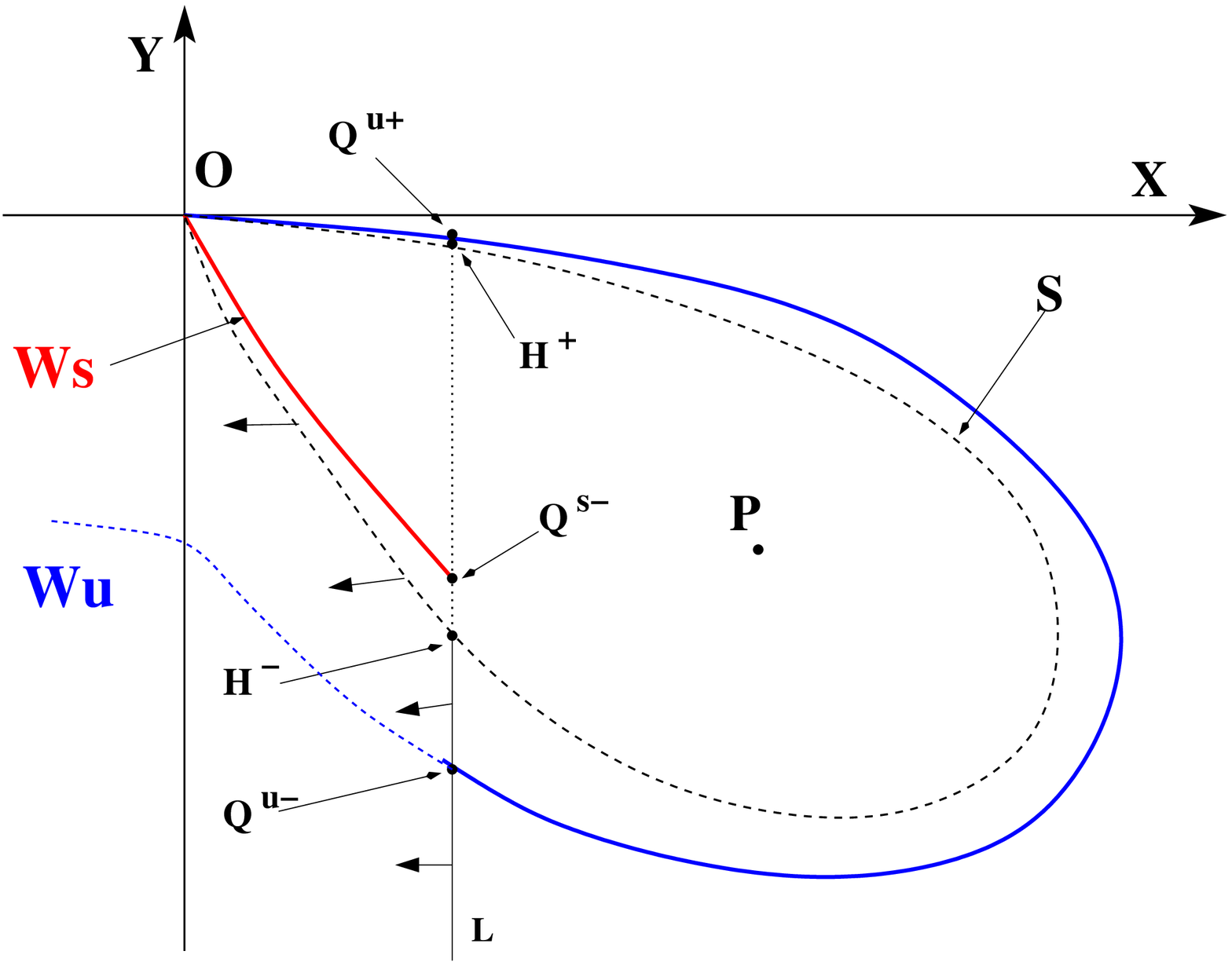}
    \caption{Sketch of the proof of Lemma \ref{picture}, when
      $\boldsymbol{A^+}$ holds.}\label{figpicture}
  \end{figure}

  Now denote by $D^u(\tau)$ the bounded set enclosed by
  $\bar{W}^u(\tau)$, the segment between $\boldsymbol{Q^{u,-}}(\tau)$
  and $\boldsymbol{H^{-}}(\tau)$ and the branch of $S(\tau)$ between
  $\boldsymbol{H^{-}}(\tau)$ and the origin: observe that by
  construction if $\Q \in D^u(\tau)$, then $\boldsymbol{y}(s,\tau,
  \Q;2^*) \in D^u(s)$ for any $s \le \tau$.  Since $S(\tau) \subset
  D^u(\tau)$ we see that if $\boldsymbol{y}(s)$ corresponds to a
  S.G.S. with slow decay,
  then   $\boldsymbol{y}(s) \in S(s)$ for any $s \in \RR$.\\
  Reasoning in the same way but reversing the direction of $s$ we see
  that if $\boldsymbol{A^-}$ holds then we can construct
  $\bar{W}^s(\tau)$ for any $\tau \in \RR$.  Denote by $D^s(\tau)$ the
  bounded set enclosed by $\bar{W}^s(\tau)$, the segment between
  $\boldsymbol{Q^{s,+}}(\tau)$ and $\boldsymbol{H^{+}}(\tau)$ and the
  branch of $S(\tau)$ between $\boldsymbol{H^{+}}(\tau)$ and the
  origin. Then if $\boldsymbol{y}(s)$ corresponds to a singular
  solution, then $\boldsymbol{y}(s) \in S(s)$ for any $s \in \RR$.  So
  Lemma \ref{picture} follows.
\end{proof}

Now we give a Lemma which is useful to detect the $\omega$-limit set
of solutions of (\ref{parab})--(\ref{data}) in the case where $\phi$
is a radial upper or lower solution of (\ref{radsta}).

\begin{lemma}\label{nosol-}
  Let $U(r)$ and $V(r)$ be positive solutions of (\ref{radsta}) either
  regular or singular and assume that there is $Z>0$ such that
  $U(Z)=V(Z)$ and $U'(Z)<V'(Z)$. Denote by
  \begin{equation}\label{defnosol}
    \zeta(r)= \left\{\begin{array}{cc}
        V(r) & \textrm{if $r \le Z$} \\
        U(r) & \textrm{if $r\ge Z$}
      \end{array}\right. \, , \qquad
    \psi(r)=  \left\{\begin{array}{cc}
        U(r) & \textrm{if $r \le Z$} \\
        V(r) & \textrm{if $r\ge Z$}
      \end{array}\right.
  \end{equation}
  Assume $\boldsymbol{A^-},\boldsymbol{Gu},\boldsymbol{Gs}$ with $l_u
  \ge 2^*$ and $l_s \in [2^*,\sigma^*)$.

  Then (\ref{radsta}) admits no solutions $\phi(r)$ either regular or
  singular such that $0<\phi(r)\le\zeta(r)$ and no solutions $\phi(r)$
  such that $\phi(r) \ge \psi(r)$ for any $r>0$.
\end{lemma}
\begin{proof}
  From Proposition \ref{super} we know that all the positive solutions
  have slow decay.  Assume first $l_s \in [2^*,\sigma^*)$. Then from
  Remark \ref{macchina} all the slow decay solutions of (\ref{radsta})
  cross each other indefinitely as $r\to +\infty$, so the Lemma easily
  follows.
\end{proof}
Reasoning in the same way we get the following:
\begin{lemma}\label{nosol+}
  Let $U(r)$, $V(r)$, $\zeta(r)$ and $\psi(r)$ be as in
  Lemma~\ref{nosol-}.  Assume
  $\boldsymbol{A^+},\boldsymbol{Gu},\boldsymbol{Gs}$ with $l_u \in
  (\sigma_*,2^*]$ and $l_s \in (2_*;2^*]$.

  Then (\ref{radsta}) admits no solutions $\phi(r)$ either regular or
  singular such that $0<\phi(r)\le\zeta(r)$ and no solutions $\phi(r)$
  such that $\phi(r) \ge \psi(r)$ for any $r>0$.
\end{lemma}

\section{Local existence} \label{sec:local-existence}  In this section
we introduce some basic facts and definitions related to the problem
\eqref{parab}--\eqref{data}, and exploiting techniques similar to
those used in
\cite[\S 1, \S 2]{W} (see also \cite[Ch. II]{QS}),
we prove local existence for the solutions of problem
\eqref{parab}--\eqref{data}.  For the remainder of this section we will
make the following assumptions, in addition to $\mathbf{F0}$, on the potential $f$ in
\eqref{parab}, i.e.:
\begin{itemize}
\item[$\boldsymbol{Fu}$:] $\boldsymbol{Gu}$ holds and there is $D>0$
  such that $\frac{\partial}{\partial y_1}[g(y_1,s;l_u)] \le D
  |y_1|^{\delta}$ for any $0 \le y_1 \le 1$, and $0< s \le 1$.\\[-0.3 cm]
\item[$\boldsymbol{Fs}$:] There are $\ell \ge 0$, $\bar{C}>0$,
  $\delta>0$, $\ep>0$ such that $|f(u,r)-f(u+h,r)| \le \bar{C} h
  u^{\delta} r^{\ell}$ whenever $r \ge 1$, $0 \le u \le 1$, and $0 \le
  h \le \ep$.
\end{itemize}
Assumption $\boldsymbol{Fu}$ is very close to $\boldsymbol{Gu}$ (and
it is actually satisfied in all the motivating examples given in the
introduction), while $\boldsymbol{Fs}$ is more standard and it is
adapted from \cite{W}.  Let us introduce the following map
\begin{equation} \label{eq:weight-w}
  w(x)= \left\{ \begin{array}{ll} |x|^{\nu} & \textrm{ if $|x|\le 1$} \\
      |x|^{\ell/\delta} & \textrm{ if $|x|\ge 1$}
    \end{array} \right. \textrm{ where } 0 \le \nu <m(l_u).
\end{equation}
We emphasize that if $\ell=0$ then $w(x) \equiv 1$ for $|x| \ge 1$.
Moreover if we set $\nu=0$ then $w(x) \equiv 1$ for $|x| \le 1$ so we
are dealing with bounded solutions, while if we set $\nu>0$ we can
deal with solutions which are unbounded for $|x|$ small and are not
defined for $x=0$.

Let us recall the definitions of \emph{continuous weak solution} and
$C_B$-\emph{mild solution} to the problem \eqref{parab}--\eqref{data}.

\begin{defin} \label{def:CW-sol} We say that a function $u$ is a
  continuous weak (c.w.) solution of \eqref{parab}--\eqref{data} if
  $u$ is continuous and it is a distributional solution: i.e.  if
  $u(x, 0) = \phi(x)$ and, for any $\eta \in C^{2,1}(\RR^n \times [0,
  T])$ with $\eta \geq 0$ and $\supp \eta(\cdot, t) \Subset \RR^n$ for
  all $t\in [0, T]$, it holds true that
  \begin{equation} \label{eq:weak-sol} \int_{\RR^n} \! u (x, s) \eta
    (x, s) dx\vert_0^{T_1}\! =\! \int_0^{T_1}\!\! \int_{\RR^n}\!\!
    \Big[ u (x, s) \big(\eta_t + \Delta \eta\big) (x, s )+ f(u, |x|)
    \eta(x, s) \Big] dx ds
  \end{equation}
  if $T_1\in [0, T]$. Further, u is a c.w. lower (respectively upper)
  solution of \eqref{parab}--\eqref{data} if $u(x, 0) \geq \phi(x)$
  (resp. $u(x, 0) \leq \phi(x)$) and we replace $``="$ in
  \eqref{eq:weak-sol} by $``\geq"$ (resp. by $``\leq"$). We call a
  function $u$ a classical solution if it satisfies
  \eqref{parab}--\eqref{data} and $u \in C^{2,1}(\RR^n \times (0, T))
  \cap C(\RR^n \times [0, T])$.
\end{defin}

Let $\phi\in C_B(\RR^n):=C(\RR^n)\cap L^\infty(\RR^n)$.  We introduce
the following operators
\begin{equation}\label{defF}
  \begin{split}
    & e^{t\Delta} \phi := (4\pi t)^{-n/2} \int_{\RR^n} \exp\Big(
    -\frac{ |x - y|^2}{4t}
    \Big) \phi(y) dy. \\
    & F_{\phi}(u)= \Big(e^{t\Delta}\phi + \int_0^t e^{(t-s)\Delta}
    f(u(\cdot, s), |\cdot |) ds\Big)(x)
  \end{split}
\end{equation}
\begin{defin} \label{def:CB-sol} We say that u is a $C_B$-mild
  solution of \eqref{parab}--\eqref{data} on $\RR^n \times [0, T)$ if
  \begin{align}
    & u \in C_B(\RR^n\times [0, T']) := C(\RR^n\times [0, T'])\cap
    L^\infty(\RR^n\times[0, T']), \,\,
    0 < T' < T, \label{eq:regularity}\\
    &u(x, t) = F_{\phi}(u(x,t)),\, \textrm{for}\, (x, t) \in \RR^n
    \times [0, T). \label{eq:mild-sol}
  \end{align}
  Also, we say that $u$ is a $C_B$-mild lower solution (or upper)
  solution if $``="$ in \eqref{eq:mild-sol} is replaced by $``\geq"$
  ($``\leq"$ respectively).
\end{defin}

We also consider the norm $\|\phi\|_X:=\|\phi w\|_{L^\infty(\RR^n)}$
and the weighted space $X = L^\infty_w(\RR^n) = \{\phi \mid \|\phi
\|_{X}< +\infty\}$ where $w$ is defined in \eqref{eq:weight-w}. We
denote by $ C_S(\RR^n\times [0, T']): =C((\RR^n\backslash\
\{0\})\times [0, T'])\cap L^\infty([0, T'], X)$ and we give the
following definiton

\begin{defin} \label{def:CS-sol} We say that $u$ is a $C_S$-mild
  solution to \eqref{parab}--\eqref{data} on $ (\RR^n\backslash\
  \{0\})\times [0, T)$ if $u \in C_S(\RR^n\times [0, T'])$ for any
  $0<T'<T$ and satisfies $u(x, t) = F_{\phi}(u(x,t))$ for $(x, t) \in
  (\RR^n\backslash\ \{0\})\times [0, T)$.
\end{defin}
Note that if $\nu>0$ and $u$ is a $C_S$-mild solution then it may be
unbounded as $x \to 0$.  Therefore we can deal with initial data and
solution having a singularity in the origin, and we will prove local
existence and uniqueness for such initial data.  It is worth recalling
that with our assumptions we have singular stationary solutions
$\phi_S$ which behave like $|x|^{-m(l_u)}$ as $x \to 0$.

 An
interesting question, still open even for the starting case
$f(u)=u^{q-1}$, is whether stationary S.G.S. are stable or not. One of
the difficulties is in fact to prove local existence and uniqueness
for nearby initial data (which is in general violated but maybe
recovered in some special space and for some parameters, see
\cite{SY5}).  In fact we cannot even hope for a general local
uniqueness result: We need to prescribe a class of function within
local uniqueness is recovered, due to the presence of self-similar
solutions converging to singular data (see, e.g. \cite{SW,QS}), and to a
new class of solutions with moving singularity recently described in
\cite{SY1,SY4}.

 We stress that here we are forced to stay below
$\phi_S$ since we need to require $\nu<m(l_u)$, so we cannot start a
stability analysis for stationary singular solutions.

The following result is a direct consequence of \cite[Lemma 1.5]{W}.
\begin{lemma} \label{lem:lemma1.5} Assume $\boldsymbol{Fu}$,
  $\boldsymbol{Fs}$.  Let $u$ be a continuous weak upper (lower)
  solution of (\ref{parab})--(\ref{data}) with $n\geq 3$.  Assume that
  there exist $k, \beta > 0$, and $0 < \alpha < 2$ such that $f(u,x) <
  k \exp(\beta |x|^\alpha{})$ on $\RR^n \times [0, T]$. Then $u(x,t)
  \geq (\leq) F_{\phi}(u(x,t)) $ on $\RR^n\times [0,T]$.
\end{lemma}

\begin{remark}
  By this lemma it follows that a c.w. solution of
  \eqref{parab}--\eqref{data} satisfying either \eqref{eq:regularity}
  in Definition \ref{def:CB-sol} or the analogous weighted condition
  in Definition \ref{def:CS-sol} is also respectively either a
  $C_B$-mild solution or a $C_S$ solution.  The converse is also true
  (see the proof of \cite[Lemma1.5]{W}).
\end{remark}
To prove Lemma \ref{lem:lemma1.5} it is sufficient to adapt the proof
of \cite[Lemma~1.5]{W} to the present case.  By Lemma
\ref{lem:lemma1.5} we also have the next result.
\begin{prop} \label{W-lemma 2.2} Assume $\boldsymbol{Fu}$,
  $\boldsymbol{Fs}$.  Suppose that $u$ is a continuous weak upper
  (resp. lower) solution of (\ref{laplace}) in $\RR^n\backslash \{0
  \}$ such that $\|u\|_X$ is bounded, then $u$ is a $C_S$-mild upper
  (resp. lower) solution of \eqref{parab}--\eqref{data}. The converse
  is also true, provided $\phi(x) \geq u(x, 0)$ (resp. $\phi(x)\leq
  u(x, 0)$). In particular if we set $\nu=0$ we see that a continuous
  bounded weak upper (resp. lower) solution of (\ref{laplace}) is a
  $C_B$-mild solution of \eqref{parab}--\eqref{data} and viceversa
\end{prop}

Take $\rho>0$ and $\phi \in X$, and denote by
$B_\rho:=B_\rho^T(0)\subseteq L^\infty([0, T]; X)$ the ball of center
$0$ and radius $\rho$, in $X$; the radius $\rho$ will be chosen
properly later.  We now prove local existence and uniqueness for $C_B$
and $C_S$-mild solutions.
\begin{lemma}\label{lemma-M}
  Assume $\boldsymbol{Fu}$, $\boldsymbol{Fs}$.  If the initial datum
  $\phi\in X$, there is $T_\phi > 0$ such that the operator
  $F_{\phi}(u)$ defined by (\ref{defF}) has a unique fixed point in
  $B_{\rho}^T(\phi)$ for any $0<T<T_\phi$, and if $T_\phi < +\infty$,
  then $\lim_{t\to T^-_\phi} \|u(\cdot, t)\|_{X} = +\infty$.
\end{lemma}
\begin{proof}
  We claim that the operator $F_{\phi}$ maps $B_\rho$ in itself and it
  is a contraction: Then the Banach fixed point theorem provides
  existence and uniqueness of a fixed point $u$ for $F_{\phi}$.

  Observe first that if $u \in B_{\rho}$, then $|u(x,t;\phi)| \le \rho
  w(x)^{-1}$ a.e. in $\RR^n \times [0,T]$. Here, we take $\rho =
  2(2D_1+2^{\nu+1}+2^{\ell/\delta})\|\phi\|_X $, where
  $D_1:=\eu^{-\nu/2}(16\nu)^{\nu/2}$. 

  From relation $\boldsymbol{Fu}$, for $|x| \le1$, we get
  \begin{align} \label{stimeFu-1} & \begin{aligned} f(u,|x|) &\le
      f(\rho\omega^{-1}(x), |x|) \le g(\rho
      |x|^{m(l_u)-\nu},\ln(|x|);l_u)
      |x|^{-(2+m(l_u))} \\
      & \le k^- |x|^{\delta(m(l_u)-\nu)-2-\nu},
    \end{aligned}\\
    \intertext{where $k^-= D \rho^{1+\delta}$. Then, for any $v \in
      B_\rho$, we also have} & \label{stimeFu-2}
    \begin{aligned}
      |f( & u,|x|)- f(v,|x|)| =
      \left|\int_0^1 \frac{\partial f}{\partial u}  (s u +(1-s) v),|x|)   [u -v] ds \right|  \\
      & \leq \frac{\partial g}{\partial
        y_1}\big(\rho|x|^{m(l_u)-\nu},\ln(|x|);l_u\big)
      |x|^{-(2+m(l_u))} | (u-v)(x)|\, |x|^{m(l_u)} \\
      & \leq \frac{\partial g}{\partial
        y_1}\big(\rho|x|^{m(l_u)-\nu},\ln(|x|);l_u\big)
      |x|^{-(2+m(l_u))} \| (u-v)(x)| \,|x|^\nu\|_{\infty} \\
      & \le k^- |x|^{\delta(m(l_u)-\nu)-2-\nu} \| u-v\|_{X},
    \end{aligned}
  \end{align}
  where we have redefined $k^-= D \max \{\rho^{1+\delta},
  \rho^{\delta} \}$.  From $\boldsymbol{Fs}$, for $|x| \ge 1$, we also
  get
  \begin{align}
    & \label{stimeFs-1}\begin{aligned} f(u,|x|) = f(u,|x|) - f(0, |x|)
      \le \bar{C} |u|^{1+\delta} |x|^{\ell} \le \bar{C}
      \rho^{1+\delta} |x|^{-\ell/\delta} \le k^+ w(x)^{-1},
    \end{aligned}\\
    & \label{stimeFs-2}
    \begin{aligned}
      |f(u,|x|)- f(v,|x|)| &\le |u|^\delta |v-u| |x|^\ell \leq \bar{C}
      \frac{\rho^{\delta}}{|x|^\ell}
      \|u-v\|_{\infty}  |x|^\ell\\
      &\le k^+ w(x)^{-1} \|u-v\|_X,\\
    \end{aligned}
  \end{align}
  with $k^+\!=\! \bar{C} \max \{ \rho^{1+\delta}, \rho^{\delta}
  \}$. Till the end of the proof we need the following straightforward
  estimate: Let $A \in (0,n)$ and denote by $\Gamma$ the Euler Gamma
  function, then:
  \begin{equation}\label{eulero}
    \int_{\RR^n}
    \frac{\eu^{-\frac{|\eta|^2}{4}}}{\big(4\pi \big)^{n/2} |\eta|^A}
    d\eta
    \leq \frac{1}{2^{n-1}(n-A)\Gamma(n/2)} +
    \int_{|\eta| \ge 1}
    \eu^{-\frac{|\eta|^2}{4}}d\eta \leq C(A) \, ,
  \end{equation}
  and we can set $C(A)=2$ if $A\in (0,n-1)$.  We now proceed to prove
  that $F_{\phi} \colon B_{\rho}\to B_\rho$. From \eqref{defF} we have
  that
  \begin{equation} \label{eq:control-on-F-phi}
    \begin{aligned}
      | F_{\phi}(u)| (x,t) & \leq |e^{t\Delta}\phi|(x) + \int_0^t
      \int_{\RR^n} \frac{\exp\Big(-\frac{|x-y|^2}{4(t-s)}\Big)}
      {\big(4\pi(t
        -s)\big)^{n/2}}f(u(y, s), |y|) dyds \\
      & \leq w^{-1}(x)\|e^{t\Delta}\phi\|_X(t) + w^{-1}(x) I,
    \end{aligned}
  \end{equation}
  where
  \begin{equation*}
    I=  w(x)\int_0^t\bigg(\int_{|y| \le 1} +\int_{|y|\geq 1}\bigg)
    \frac{\exp\Big(-\frac{|x-y|^2}{4s}\Big)}{\big(4\pi s
      \big)^{n/2}} f\big(u, |y|\big) dyds =:I_a+I_b.
  \end{equation*}
  Using \eqref{stimeFu-1}-\eqref{stimeFu-2} we get the following
  \begin{equation}\label{K-}
    \begin{aligned}
      I_{a} &\leq k^- \int_0^t\bigg(\int_{\frac{|x|}{2} \le |y| \le 1}
      +\int_{|y|\le \frac{|x|}{2}}\bigg) \frac{w
        (x)\exp\Big(-\frac{|x-y|^2}{4s}\Big)}{\big(4\pi s \big)^{n/2}
        |y|^{2-\delta (m(l_u)-\nu)+\nu}}
      dyds \\
      & =: k^- (I_a^-+I_a^+)
    \end{aligned}
  \end{equation}
  and
  \begin{align*}
    I_a^- &\le 2^{\nu} \int_0^t \int_{\frac{|x|}{2} \le |y| \le 1}
    \frac{ 
      \exp\Big(-\frac{|x-y|^2}{4s}\Big) }{\big(4\pi s
      \big)^{n/2} |y|^{2-\delta (m(l_u)-\nu)}} dyds\\
    &\le 2^{\nu} \int_0^t \int_{ \RR^n} \frac{ 
      \exp\Big(-\frac{|y|^2}{4s}\Big)}{\big(4\pi s \big)^{n/2}
      |y|^{2-\delta (m(l_u)-\nu)}}dyds\\
    &\le 2^{\nu} \int_0^t \int_{\RR^n} \frac{ 
      \exp(-|\eta|^2) |\eta|^{\delta (m(l_u)-\nu)-2} }{ (\pi )^{n/2}
      s^{1-\delta (m(l_u)-\nu)/2}} d\eta ds \le
    K^- 
    t^{\delta (m(l_u)-\nu)/2 },
  \end{align*}
  where $K^->0$ is a 
  positive constant, and we used the fact that the convolution of
  radial decreasing function is radial decreasing too (see
  \cite[Lemma~1.4]{W}), and that $n-3+\delta (m(l_u)-\nu) >-1$.

  Since $|x-y| \ge \big||x|-|y| \big|$ we get
  \begin{align*}
    \,\,\, I_a^+ &\le \int_0^t \int_{|y| \le \frac{|x|}{2} } \frac{
      w(x)\exp\Big(-\frac{|x|^2}{16 s}
      \Big)}{\big(4\pi s \big)^{n/2}  |y|^{2 + \nu -\delta (m(l_u)-\nu)}}dyds  \\
    & \le \int_0^t \int_{|y| \le \frac{|x|}{2} } \frac{
      w(x)\exp\Big(-\frac{|x|^2}{32 s} \Big)\exp\Big(-\frac{|y|^2}{8
        s}\Big)}{\big(4\pi s \big)^{n/2}
      |y|^{2+\nu-\delta (m(l_u)-\nu)}} dyds \\
    & \le \int_0^t \frac{w(x)\exp\Big(-\frac{|x|^2}{32 s}\Big)}{ s^{1
        -[ \delta (m(l_u)-\nu)/2] + \nu/2} }ds \int_{\RR^n}
    \frac{\exp\Big(-\frac{|\eta|^2}{8}\Big)}{
      \big(4\pi \big)^{n/2}|\eta|^{2+\nu-\delta (m(l_u)-\nu)} }d\eta \\
    &\leq 2^{n/2} C\int_0^t \frac{w(x)}{ s^{1 -\delta m(l_u)/2+\nu/2}}
    \exp\Big(-\frac{|x|^2}{32 s}\Big) ds
  \end{align*}
  where we used that $2+\nu -\delta (m(l_u)-\nu)<2+ m(l_u)<n$, and
  $C=C(2+m(l_u))$ is the constant defined in (\ref{eulero}). Now we
  need to distinguish between the $|x| \le 1$ and the $|x| \ge 1$
  case.  Assume the former so that $w(x)=|x|^{\nu}$, and observe that
  $h(a):=\eu^{-a/32}a^{\nu/2} \le D_1=h(16 \nu)$; for any $a \ge 0$ we
  get
  \begin{equation}\label{K+}
    \begin{aligned}
      I_a^+\leq D_1 C\int_0^t \frac{1}{ s^{1 -[ \delta
          (m(l_u)-\nu)/2]}}ds \le K^+ t^{[ \delta (m(l_u)-\nu)/2]}.
    \end{aligned}
  \end{equation} where
  $K^+ =\frac{D_1 C}{\delta (m(l_u)-\nu)/2}>0$ is a constant. When $|x| \ge 1$ so that $w(x)=|x|^{\ell/\delta}$;
  setting $\bar{h}(a):=\eu^{-a/64}a^{\ell/2\delta}$ we find $\bar{h}(a) \le D_2:=\bar{h}(32 \ell/\delta)$ for any $a \ge 0$
  and similarly $\tilde{h}(a):=\eu^{-a/64}a^{1 -[ \delta (m(l_u)-\nu -1/\ell)/2] + \nu/2}  $ is bounded, say $\tilde{h}(a) \le D_3$. Thus
  \begin{equation}\label{K+I}
    \begin{aligned}
      I_a^+ & \le C\int_0^t
      \left[\left(\frac{|x|^2}{s}\right)^{\frac{\ell}{2\delta}}
        \exp\Big(-\frac{|x|^2}{64s}\Big)\right] \frac{
        \exp\Big(-\frac{1}{64 s}\Big)}{ s^{1 -[ \delta (m(l_u)-\nu
          -1/\ell)/2] + \nu/2} }ds \\ & \le C D_2 \int_0^t
      \bar{h}(\frac{|x|^2}{s})\tilde{h}(\frac{1}{s}) ds \le CD_2D_3 t
      \le \, K^+t
    \end{aligned}
  \end{equation}
  with a possibly larger constant $K^+$.  Now we estimate $I_b$. From
  \eqref{stimeFs-1}--\eqref{stimeFs-2} we get:
  \begin{equation}\label{Ib}
    \begin{aligned}
      & I_{b} \leq k^+ w(x) \int_0^t\bigg(\int_{ |y| \ge \max
        \{\frac{|x|}{2},1 \}}+\int_{1 \le |y| \le \frac{|x|}{2}
      }\bigg) \frac{\exp\Big(-\frac{|x-y|^2}{4s}\Big)}{\big(4\pi s
        \big)^{n/2} |y|^{\ell/\delta} }
      dyds  \\
      &=:k^+ (I_b^-+I_b^+) \, .
    \end{aligned}
  \end{equation}
  Observing that $\hat{h}(a)= a^{\ell/(2 \delta)} \eu^{-a/32} \le
  \hat{h}(8 \ell/\delta)$ and $C_b=2^{3n/2} \hat{h}(8 \ell/\delta)$,
  we find
  \begin{align*}
    & I_b^- \le w(x) \int_0^t \int_{ |y| \ge \max \{\frac{|x|}{2},1
      \}} \frac{\exp\Big(-\frac{|x-y|^2}{4s}\Big)}{\big(4\pi s
      \big)^{n/2}|y|^{\ell/\delta}}dyds \le  \int_0^t\frac{w(x)}{w(\max \{\frac{|x|}{2},1\})}    \le 2^{\ell/ \delta}  t,\\
    &
    \begin{aligned}
      I_b^+ &\le \int_0^t w(x) \eu^{-\frac{|x|^2}{32 s} } ds \int_{ 1
        \le |y| \le \frac{|x|}{2} } \frac{ \eu^{-\frac{|x|^2}{32 s} }
      }{\big(4\pi s \big)^{n/2}}dy\le 2^{3n/2} |x|^{\ell/\delta} t
      \eu^{-\frac{|x|^2}{32 t} } \le C_b t^{1+\frac{\ell}{2\delta}} ,
    \end{aligned}
  \end{align*}
  Therefore, there is $K>0$ such that $I \le K \max
  \{t^{1+\ell/(2\delta)},t , t^{\delta(m(l_u)-\nu)/2} \}$, with $K=
  K(n, \ell, \nu, \|\phi\|_{X})$ and relation
  \eqref{eq:control-on-F-phi} reduces to
  \begin{equation} \label{eq:ball-into-ball} |F_{\phi}(u)|(x, t) \leq
    w^{-1}(x)\| e^{t\Delta}\phi \|_{X} \!+  w^{-1}(x) K\max
    \{T^{1+\ell/(2\delta)}, T, T^{\delta(m(l_u)-\nu)/2}\}.\!\!\!
  \end{equation}
  To estimate the term $w^{-1}(x) \|e^{\Delta t}\phi \|_X $ we follow
  an approach similar to the one used above. Indeed, we rewrite $
  e^{t\Delta}\phi (x)$ as follows
  \begin{equation} \label{eDeltaT} e^{t\Delta}\phi (x)=
    \bigg(\int_{|y| \le 1}\!\! +\int_{|y|\geq 1}\!\!\bigg)
    \frac{\exp\Big(-\frac{|x-y|^2}{4t}\Big)}{\big(4\pi t \big)^{n/2}}
    \phi(y) dy =: I_\alpha + I_\beta.
  \end{equation}
  Hence, we get
  \begin{align*}
    I_\alpha &\leq \|\phi\|_X \bigg(\int_{\frac{|x|}{2} \le |y| \le 1}
    +\int_{|y|\le \min\{ \frac{|x|}{2},1 \}}\bigg) \frac{
      \exp\Big(-\frac{|x-y|^2}{4 t} \Big)}{\big(4\pi t
      \big)^{n/2}w(y)} dy =: I_{\alpha}^-+I_{\alpha}^+.
  \end{align*}
  For $\frac{|x|}{2} \le |y| \le 1$, we reach
  \begin{equation}\label{KI-alpha-}
    \begin{aligned}
      I_{\alpha}^-&\leq \frac{\|\phi\|_X 2^\nu}{w(x)}
      \int_{\frac{|x|}{2}\leq |y| \leq 1}
      \frac{\exp\Big(-\frac{|x-y|^2}{4 t}\Big)
      }{\big(4\pi t \big)^{n/2}} dy\\
      & \leq \frac{\|\phi\|_X}{w(x)} 2^\nu \int_{\RR^n}
      \frac{\exp\Big(-\frac{|y|^2}{4 t}\Big)}{\big(4\pi t \big)^{n/2}
      } dy = \frac{\|\phi\|_X 2^\nu}{w(x)} .
    \end{aligned}
  \end{equation}
  For the term $I_\alpha^+$, using (\ref{eulero}) we obtain
  \begin{equation*} \label{I-alpha+}
    \begin{aligned}
      I_{\alpha}^+&\leq \|\phi\|_X \exp\Big( \frac{-|x|^2}{32t}\Big)
      \int_{|y| \le \min\{ \frac{|x|}{2}, 1\}}
      \frac{\exp\Big(-\frac{|y|^2}{8 t}\Big)}{\big(4\pi t
        \big)^{n/2}|y|^\nu }
      dy \\
      &\leq \|\phi\|_X \exp\Big( \frac{-|x|^2}{32t}\Big) \int_{\RR^n}
      \!\!\!  \frac{\exp\Big(-\frac{|\eta|^2}{4 }\Big) }{\big(4\pi
        \big)^{n/2} |\eta|^\nu t^{\nu/2}} d\eta \le 2 \|\phi\|_X
      \frac{\exp\Big( \frac{-|x|^2}{32t}\Big)}{t^{\nu/2}} .\!\!\!\!\!
    \end{aligned}
  \end{equation*}
  Now, arguing as in (\ref{K+}), (\ref{K+I}) we find for $|x| \le 1$
  \begin{equation} \label{KI-alpha+}
    \begin{aligned}
      I_{\alpha}^+& \le 2 \|\phi\|_X
      \frac{\Big(\frac{|x|^2}{t}\Big)^{\nu/2}\exp\Big(
        \frac{-|x|^2}{32t}\Big)}{|x|^{\nu}} \le \frac{2D_1\|\phi\|_X
      }{w(x)}
    \end{aligned}
  \end{equation}
  while for $|x| \ge 1$, since $\bar{h}(a)=a^{\ell/(2\delta)}
  \eu^{-a/64} \le D_2$, we get
  \begin{equation*}
    \begin{aligned}
      I_{\alpha}^+& \le 2 \|\phi\|_X
      \frac{\Big(\frac{|x|^2}{t}\Big)^{\ell/(2\delta)}\exp\Big(
        \frac{-|x|^2}{64t}\Big)}{w(x)} \; \frac{\exp\Big(
        \frac{-|x|^2}{64t}\Big)}{t^{\nu/2-\ell/(2\delta)}} \le \ep(t)
      \; \frac{2D_2\|\phi\|_X }{w(x)}
    \end{aligned}
  \end{equation*}
  where $\ep(t):=\exp\big( \frac{-1}{64t}\big)t^{\ell/(2\delta)-nu/2}
  \to 0$ as $t \to 0$. So choosing $t$ small enough we can assume that
  (\ref{KI-alpha+}) holds for $|x| \ge 1$, too.  Take into account
  $I_\beta$, to get
  \begin{equation}\label{KIbeta}
    \begin{split}
      I_\beta &\leq \frac{\|\phi\|_X}{w(x/2)} \int_{ |y| \ge
        \frac{|x|}{2} }
      \frac{\exp\Big(-\frac{|x-y|^2}{4t}\Big)}{\big(4\pi t
        \big)^{n/2}} \\
&\leq \frac{\|\phi\|_X}{w(x/2)}\|\phi\|_X
      \int_{\RR^n} \frac{\exp\Big(-\frac{|\eta|^2}{4}\Big)}{\big(4\pi
        \big)^{n/2}}d\eta \\
      & \leq \frac{\|\phi\|_X}{w(x/2)} \le
      \frac{\|\phi\|_X}{w(x)} \max \{2^{\nu}, 2^{\ell/\delta} \}
    \end{split}
  \end{equation}
  Collecting the estimates
  \eqref{eDeltaT}-\eqref{KI-alpha-}-\eqref{KI-alpha+}-\eqref{KIbeta},
  we have that relation \eqref{eq:ball-into-ball}, and hence
  \eqref{eq:control-on-F-phi}, for $T \le 1$ gives
  \begin{equation*}
    \begin{aligned}
      \| F_{\phi}(u) \|_X(T) &\leq \big( 2D_1+
      2^{\nu+1}+2^{\ell/\delta} \big)
      \|\phi\|_X +  K\max \{ T,   T^{\delta(m(l_u)-\nu)/2}\}\\
      &\leq \rho/2 + CT\|\phi\|_X + K\max \{ T,
      T^{\delta(m(l_u)-\nu)/2}\}.
    \end{aligned}
  \end{equation*}
  Let $T_0 = T_0(n, \ell, \|\phi\|_{X}, \nu, \rho) > 0$ be such that,
  for any $T\leq T_0$,
  \begin{equation*}
    K\max \{ T,   T^{\delta(m(l_u)-\nu)/2}\}< \rho/2.
  \end{equation*}
  Then, we have that
  \begin{equation*}
    \| F_{\phi}(u)\|_{L^\infty(0, T; X)}\leq \rho, \textrm{ for any } T\leq T_0
  \end{equation*}
  and hence $F_{\phi}$ maps $B_\rho$ into $B_\rho$, for $T \leq T_0$.


  \medskip

  Analogously, let $u,v \in B_{\rho}$, we get
  \begin{equation}
    \begin{split}
      |F_{\phi}(u)-F_{\phi}(v)|(x, t)  \leq & \int_0^t \Big(\int_{|y|
        \le 1} \frac{k^- \exp\Big(-\frac{|x-y|^2}{4 s}\Big)}{\big(4\pi
        s \big)^{n/2} |y|^{2+\nu-\delta (m(l_u)-\nu)}} dy +
      \\
      & +\int_{|y|\ge 1} \frac{k^+\exp\Big(-\frac{|x-y|^2}{4
          s}\Big)}{\big(4\pi s \big)^{n/2}|y|^{\ell/\sigma}} dy \Big)
      \| u-v\|_{X} ds.
    \end{split}
  \end{equation}
  Repeating the argument of (\ref{K-}) and (\ref{K+}) we get
  \begin{equation}\label{contracta}
      \| [F_{\phi}(u)-  F_{\phi}(v)]\|_X (t) \le k^-[K^- t^{\delta
        (m(l_u)-\nu)/2}+ K^+ t]+k^+[2^{\ell/ \delta} + C_b
      t^{\ell/(2\delta)}] t.
  \end{equation}
  Therefore, taking $T < T_0$ sufficiently small, it follows that
  $F_{\phi}$ maps $B_\rho$ into $B_\rho$ and it is actually a
  contraction. From the contraction principle, we obtain existence and
  uniqueness of a fixed point $u$ in $B_\rho$ which in turn implies
  the existence and local uniqueness of a $C_B$-mild solution to
  \eqref{parab}--\eqref{data}.  Then, we can restart the reasoning, by
  setting $\phi(x) = u(x, T)$ and go up to $T_\phi$ by a ladder
  argument.  Note that if $T_{\phi}<\infty$ and $\lim_{t \to T_{\phi}}
  \|u(x,t)\|_X$ is bounded we can restart the ladder argument and
  obtain a continuation interval $[0,T'] \supset [0,T_{\phi})$ and
  this is a contradiction. Hence if $T_{\phi}<\infty$ we get $\lim_{t
    \to T_{\phi}} \|u(x,t)\|_X=+\infty$.
\end{proof}


As a direct consequence of Lemma~\ref{lemma-M} we have the following
existence result

\begin{thm} \label{teo:CB-CS-mild-sol} Assume $\boldsymbol{Fu}$,
  $\boldsymbol{Fs}$. Let $\phi\in X$ be the initial datum for the
  Equation \eqref{parab}. Then problem \eqref{parab}--\eqref{data} has
  a unique weak solution $u$ on $\RR^n \times [0, T_\phi)$,
  $T_\phi>0$.  If $\phi\in C_B(\RR^n)=C(\RR^n)\cap L^\infty(\RR^n)$
  then $u\in C\big(\RR^n\times [0, T_\phi)\big)\cap
  L^\infty\big(\RR^n\times[0, T_\phi)\big)$, and if $T_\phi < +\infty$
  then $\lim_{t\to T_\phi^-} \|u(\cdot , t)\|_{L^\infty(\RR^n)} =
  +\infty$.  Similarly, if $\phi\in C_S(\RR^n) =C(\RR^n\backslash
  \{0\})\cap L^\infty_w(\RR^n) $ then $u\in
  C\big((\RR^n\backslash\{0\})\times [0, T_\phi)\big) \cap
  L^\infty([0, T_\phi), L^\infty_w(\RR^n) )$, and if $T_\phi <
  +\infty$, then $\lim_{t\to T_\phi^-} \|u(\cdot,
  t)w(\cdot)\|_{L^\infty(\RR^n)} = +\infty$.

  Furthermore, if $\phi\geq 0$, then $u \geq 0 $; if $\phi$ is radial,
  then $u$ is radial in $x$; if $\phi$ is radial and radially
  non-increasing, then $u$ is non-increasing in $|x|$.
\end{thm}

\begin{remark}
  Assume that $f$ is locally H\"older
  continuous in $x$ and assume also there exists $l\geq 0$ such that
  $f(u, |x| ) |x|^{-l} = f(u, r)r^{-l}$ is
  locally Lipschitz in $u$ uniformly with respect to $x$ (and so in
  $r$) in any bounded subset of $\RR^n$.  In such a case, using
  \cite[Lemma~1.2]{W} and arguing as in \cite{W}, one can verify that
  the $C_B$-solutions are actually classical.

  In particular, we have that the potentials in
  \eqref{eq:potential-0}--\eqref{eq:potential-2.5} verify these
  conditions.
  \end{remark}

As a consequence of Proposition~\ref{W-lemma 2.2} and
Lemma~\ref{lemma-M} we 
have the following result
\begin{thm} \label{thm:theorem-2.4} Assume that $\boldsymbol{Fu}$,
  $\boldsymbol{Fs}$ are verified. Then
  \begin{itemize}
  \item [\textbf{(i)}:] Suppose that $\bar{u}$ and $\underbar{u}$ are $C_S$-mild
    upper and lower solutions of \eqref{parab}-\eqref{data} on $\RR^n
    \times [0, T)$.  Then $\bar{u }> \underbar{u}$ on $\RR^n \times
    [0, T)$, and the unique $C_S$-mild solution of
    \eqref{parab}-\eqref{data} on $\RR^n \times [0, T_\phi)$ satisfies
    that $\underbar{u}\leq u \leq \bar{u}$ on $\RR^n \times [0, T)$
    and $T_\phi>T$.\\[-0.3 cm]
  \item [\textbf{(ii)}:] If the initial value $\phi$ in \eqref{data} is a
    c.w. upper (lower) solution of (\ref{laplace}), then the
    $C_S$-mild solution u of \eqref{parab}-\eqref{data} is
    non-increasing (non-decreasing) in $t \in [0, T_\phi)$.\\[-0.3 cm]
  \item[\textbf{(iii)}:] If $\phi$ is radial, then $\bar{u}$, $\underbar{u}$ and
    $u$ are radial for any $t$ in their dominion of definition.\\[-0.3
    cm]
  \item[\textbf{(iv)}:] If $\phi$ is a c.w. upper (lower) solution but not a
    solution of (\ref{laplace}), then $u_t(x, t) < (>) 0$, $t>0$.
  \end{itemize}
\end{thm}

The proof of this theorem is omitted because it can be easily derived
by adapting that one of \cite[Theorem~2.4]{W} to the current case.  We
point out that claim \textbf{(iv)} follows directly by exploiting a comparison
principle, arguing as in \cite[Lemma~2.6]{W} (see, e.g.,
\cite{Gazzola} for a full-fledged proof of this well-established
comparison argument. See also \cite{BH}).  Further, a result analogous
to Theorem~\ref{thm:theorem-2.4} holds true also in the even simpler
case of the $C_B$-mild solutions of \eqref{parab}--\eqref{data} on
$\RR^n \times [0, T)$ (see \cite[Theorem~2.4]{W}).

To conclude this section we give a result about the global solution
$u(x, t;\phi)$ of the problem (\ref{parab})--(\ref{data}), for
$\phi\in C_B(\RR^n)$ or $\phi\in C_S(\RR^n)$.

%

%
\begin{remark}\label{limit} Assume that $\phi$ is a
  singular upper (respectively lower) solution.  From
  Theorem~\ref{thm:theorem-2.4} point $\textbf{(iv)}$, which
  translates \cite[Lemma~2.6]{W} to the current case, it follows that
  $\lim_{t\to T_\phi} u(x, t;\phi) = u(x, T_\phi;\phi)$ exists for any
  $x \ne 0$.  Following the proof of Claim 2 of \cite[Theorem~3.6]{W},
  using Lebesgue dominated convergence theorem and regularity theory
  for elliptic equation we see that $u(x, T_\phi; \phi)$ is a
  distributional solution of (\ref{laplace}). Moreover if $\phi$ is
  radial then $u(x, T_\phi; \phi)$ is radial too.
\end{remark}

\section{Long time behavior: main results}
\label{sec:long-time-behavior}
Now we are ready to state and prove our results 
in their general form, from which Theorems \ref{wangslowLM},
\ref{main2}, and Corollary~\ref{main3} follow.

Let $w(x)$ be defined as in (\ref{eq:weight-w}).
\begin{thm}\label{unstable}
  Assume $\boldsymbol{Fu}$, $\boldsymbol{Fs}$, $\boldsymbol{A^-}$,
  $\boldsymbol{Gu}$ and $\boldsymbol{Gs}$, with $2^* \le l_s <
  \sigma^*$ and $l_u \ge 2^*$.
  \begin{itemize}
  \item [\textbf{(i)}:] If $\phi(x) \lneqq U(|x|,\alpha)$ for some $\alpha>0$,
    then it holds that $\| u(x, t ; \phi) \|_{\infty} \to 0$ and $\|
    u(x, t ; \phi)(1+|x|^{\nu}) \|_{\infty} \to 0$ as $t \to +\infty$
    for any $0 \le \nu < m(l_s)$.\\[-0.3 cm]
  \item [\textbf{(ii)}:] Let $\phi$ be a continuous initial data or a singular
    one such that $\|\phi(x)w(x)\|_{\infty}$ is bounded, for some $\nu
    \in [0,m(l_s))$. If $\phi(x) \gneqq U(|x|,\alpha)$ for some
    $\alpha>0$, then $\|u(x, t ; \phi)
      w(x)\|_{\infty} $ must blow up in finite time.
  \end{itemize}
\end{thm}
This result establishes that G.S. with s.d. are on the border between
the basin of attraction of the null solutions and initial data which
blow up in finite time, in the range of the considered parameters. We
emphasize that we need $l_s< \sigma^*$, but $l_u$ has no upper bound:
this allows e.g. $f(u,r)=(1+r^a)u^{q}$ where $q>\sigma^*$ and $a \in
\big(\frac{2}{\sigma^*-2}(q-\sigma^*) , \frac{n-2}{2}(q-2^*) \big)$.
It is easy to check that $\boldsymbol{A^-}$ might be replaced by
$\boldsymbol{H-}$,   and
that Theorem \ref{unstable} implies Theorem \ref{second}.

Next, we state Theorems \ref{unstableslow} and \ref{unstableadd1}
from which Theorems \ref{wangslowLM}, \ref{main2}, and Corollary
\ref{main3} follow.  These results concern a wider range of
parameters, and enable us to understand something about the border of
the basin of attraction of the null solution and of infinity,
i.e. blowing up solutions.  We start from a generalization of a result
by Wang in \cite{W} concerning slow decay solutions.
\begin{thm}\label{unstableslow}
  Assume $\boldsymbol{Fu}$, $\boldsymbol{Fs}$, $\boldsymbol{Gu}$,
  $\boldsymbol{Gs}$.  Assume either $\boldsymbol{A^+}$ with $l_u,l_s
  \in (2_*, 2^*]$ or $\boldsymbol{A^-}$ with $l_u \ge 2^*$ and $2^*
  \le l_s <\sigma^*$.  Then there exists a one-parameter family of
  upper radial solutions of (\ref{laplace}) denoted by
  $\chi_{\tau}(x)$, such that $u(x,t;\chi)$ converge to the null
  solution as $t \to +\infty$ and with the properties described in
  Theorem \ref{wangslowLM}; in particular they have slow decay.
\end{thm}
Such a result is proved in \cite{W} for $f(u,r)=|x|^{\delta} u^{q-1}$,
but as far as we are aware is new even for the potential considered in
\cite{DLL} and \cite{YZ}, i.e. even for $f$ as in
(\ref{eq:potential-0}) with $k(r) \ne r^{\delta}$ and $f$ as in
(\ref{eq:potential-1}) also when $k_1=k_2=1$.  Theorem
\ref{unstableslow} seems to suggest that slow decay is the optimal
decay rate to have solutions of (\ref{parab}) defined for any $t$. In
fact we have the following result which is in contradiction with this
idea (as far as we are aware these results are new even for the case
$f(u)=u^{q-1}$).
\begin{thm}\label{unstableadd1}
  Assume $\boldsymbol{Fu}$, $\boldsymbol{Fs}$, $\boldsymbol{Gu}$,
  $\boldsymbol{Gs}$.  Assume either $\boldsymbol{A^+}$ with $l_u,l_s
  \in (2_*, 2^*]$ or $\boldsymbol{A^-}$ with $l_u,l_s \ge 2^*$.  Then
  there are one-parameter families of upper and lower radial solution
  of (\ref{laplace}), denoted by $\zeta_{\tau}(x)$ and
  $\psi_{\tau}(x)$ having the properties described in Theorem
  \ref{main2}.
\end{thm}
We stress that Remark \ref{unstableadd3} holds also in this case.\\
The relevance of Theorems \ref{unstableslow} and \ref{unstableadd1} is
more clear if we recall the comparison principle: if we choose any
$\phi(x)$ even non radial, such that $\phi \le\not\equiv \zeta_{\tau}$
for some $\tau$ then $u(x,t;\phi)$ converge to the null solution,
while if $\phi \ge\not\equiv \psi_{\tau}(x)$ then $u(x,t;\phi)$ blows
up in finite time. In fact Corollary \ref{main3} holds in this more
general context too.

\subsection{Construction of upper and lower-solutions for the stationary problem.}
From now to the end we always assume
$\boldsymbol{F0},\boldsymbol{Fu},\boldsymbol{Fs}$ (in order to
guarantee local existence of solutions) without further mentioning.
We introduce the following notation; we denote by
$\boldsymbol{y^u}(s,\alpha;l_u)$ the trajectory of (\ref{si.na})
corresponding to the regular solution $U(r,\alpha)$ of (\ref{radsta}),
by $\boldsymbol{y^u}(s,\infty;l_u)$ the trajectory corresponding to
the singular solution $U(r,\infty)$, by
$\boldsymbol{y^s}(s,\beta;l_s)$ the trajectory corresponding to the
fast decay solution $V(r,\beta)$, and by
$\boldsymbol{y^s}(s,\infty;l_s)$ the trajectory corresponding to the
slow decay solution $V(r,\infty)$.
\begin{lemma}\label{uniforme1}
  Fix $S \in \RR$. Assume $l_u>2^*$, then
  $\boldsymbol{y^u}(s,\alpha;l_u)$ converges to
  $\boldsymbol{y^u}(s,\infty;l_u)$ as $\alpha \to +\infty$, uniformly
  for
  $s \le S$.\\
  Assume $2_*<l_s<2^*$, then $\boldsymbol{y^s}(s,\beta;l_s)$ converges
  to $\boldsymbol{y^s}(s,\infty;l_s)$, uniformly for $s \ge S $.
\end{lemma}
We think it is worth recalling that in the cases considered
$U(r,\infty)$ and $V(r,\infty)$ are the unique singular and slow decay
solutions of (\ref{radsta}).
\begin{proof}[Proof of Lemma~\ref{uniforme1}]
  Assume $l_u>2^*$, fix $\tau \in \RR$ and set
  $\Q(\alpha)=\boldsymbol{y^u}(\tau,\alpha;l_u)$. We recall that
  $W^u(\tau;l_u)$ is a path joining the origin and
  $\R=\boldsymbol{y^u}(\tau,+\infty;l_u)$; moreover $\Q(\alpha) \to
  \R$ as $\alpha \to \infty$, see Remark \ref{corr} and \cite{Fcamq}.
  If $I \subset \RR$ is a compact interval, then
  $\boldsymbol{y^u}(t,\alpha;l_u) \to
  \boldsymbol{y^u}(\tau,+\infty;l_u)$ for any $t \in I$. But
  $\boldsymbol{y^u}(t,\alpha;l_u)$ and
  $\boldsymbol{y^u}(\tau,+\infty;l_u)$ are solutions of (\ref{si.na})
  hence $\boldsymbol{y^u}(t,\alpha;l_u)$ is equibounded and
  equicontinuous for $\alpha$ large so we conclude using Ascoli
  theorem.  The case $2_*<l_s<2^*$ is completely analogous.
\end{proof}
Analogously, using the fact that $\boldsymbol{y^u}(s,\alpha ;l_u) \to
(0,0)$ as $s \to -\infty$, and $\boldsymbol{y^s}(s,\beta ;l_s) \to
(0,0)$ as $s \to +\infty$ for any $\alpha>0$ $\beta>0$, we get the
following
\begin{lemma}\label{uniforme2}
  Assume $\boldsymbol{Gu}$ with $l_u>2_*$, and fix $S \in \RR$; then
  the trajectory $\boldsymbol{y^u}(s,\alpha_2;l_u)$ converges to
  $\boldsymbol{y^u}(s,\alpha_1;l_u)$ as $\alpha_2 \to \alpha_1$,
  uniformly for
  $s \le S$.\\
  Assume $\boldsymbol{Gs}$ with $l_s>2_*$, and fix $S \in \RR$; then
  $\boldsymbol{y^s}(s,\beta_2;l_s)$ converges to
  $\boldsymbol{y^s}(s,\beta_1;l_s)$ as $\beta_2 \to \beta_1$,
  uniformly for $s \ge S$.
\end{lemma}
From Lemmas \ref{uniforme1} and \ref{uniforme2} we easily get the
following.
\begin{lemma}\label{uniforme}
  Let $\rho>0$ be arbitrarily small.  Assume $\boldsymbol{Gu}$,
  $\boldsymbol{Gs}$ with $l_u, l_s \ge 2^*$, and $\boldsymbol{A^-}$.
  Then for any $\alpha_1 \ge 0$, $U(r,\alpha_2)$ converges to
  $U(r,\alpha_1)$ as $\alpha_2 \to \alpha_1$, uniformly for $ r \ge
  0$. Further, assume $l_u>2^*$, then $U(r,\alpha_2)$ converges to
  $U(r,\infty)$ as $\alpha_2 \to +\infty$, uniformly for
  $ r \ge \rho$.\\
  Similarly assume $\boldsymbol{Gu}$, $\boldsymbol{Gs}$ with $l_u, l_s
  \in (2_*,2^*]$, and $\boldsymbol{A^+}$.  Then for any $\beta_1 \ge 0
  $, $V(r,\beta_2)$ converges to $V(r,\beta_1)$ as $\beta_2 \to
  \beta_1$, uniformly for $ r \ge \rho$.  Moreover, if $2_*<l_s<2^*$
  then we can take also $\beta_1=+\infty$.
\end{lemma}
\begin{proof}
  Assume $\boldsymbol{Gu}$, $\boldsymbol{Gs}$ with $l_u, l_s \ge 2^*$,
  and $\boldsymbol{A^-}$ and choose $\alpha_1 \in (0,+\infty)$.  Then
  from Lemma \ref{uniforme2} we easily see that $U(r,\alpha_2)$
  converges to $U(r,\alpha_1)$ uniformly for $r$ in compact subsets of
  $r \in (0,\infty)$. However using Ascoli theorem and working
  directly on (\ref{radsta}) we easily get uniform convergence for $r
  \in [0,\rho)$ too, see e.g. the Appendix of \cite{FLS} for more
  details (in a much more general context); so we have uniform
  convergence in $[0,R]$.  From Proposition \ref{super} we know that
  $U(r,\alpha)$ is a G.S. with slow decay for any $\alpha>0$, hence
  $y^u_1(s,\alpha;l_s)$ is positive and bounded for $s \ge 0$.  Thus
  setting $r= \eu^{s} \ge 1$ we find
$$| U(r,\alpha_2)-U(r,\alpha_1)|
=|y^u_1(s,\alpha_2;l_s)-y^u_1(s,\alpha_1;l_s)| \eu^{-m(l_s) s} < K
r^{-m(l_s)} \, . $$ So for any $\ep>0$ we can choose $R_0=
(K/\ep)^{1/m(l_s)}$ so that $|U(r,\alpha_2)-U(r,\alpha_1)|< \ep$ for
$r>R_0$. Then we can choose $R>R_0$ and we have that $U(r,\alpha_2)$
converges to $U(r,\alpha_1)$ uniformly for $r \ge 0$.

When $l_u>2_*$ we simply repeat the argument for $U(r,\infty)$ using
Lemma~\ref{uniforme1} instead of Lemma \ref{uniforme2}.\\
Now, assume $l_u, l_s \in (2_*,2^*]$ and that $\boldsymbol{A^+}$
holds, so that $V(r,\beta_1)$ is a S.G.S. with fast decay. Then we get
uniform convergence for $r \ge \rho$ directly from Lemma
\ref{uniforme2} and if $l_s<2^*$, we conclude by using Lemma
\ref{uniforme1}.
\end{proof}

\begin{lemma}\label{gs-}
  Assume $\boldsymbol{A^-},\boldsymbol{Gu},\boldsymbol{Gs}$ with $l_u
  \ge 2^*$ and $l_s \in [2^*,\sigma^*)$, then for any
  $0<\alpha_1<\alpha_2 \le \infty$ there is $Z(\alpha_2,\alpha_1)>0$
  such that $U(r,\alpha_2)>U(r,\alpha_1)$ for
  $0<r<Z(\alpha_2,\alpha_1)$ and $U(r,\alpha_2)=U(r,\alpha_1)$,
  $U'(r,\alpha_2)<U'(r,\alpha_1)$ for
  $r=Z(\alpha_2,\alpha_1)$. \\
  Assume $\boldsymbol{A^+},\boldsymbol{Gu},\boldsymbol{Gs}$ with
  $\sigma_* <l_u \le 2^*$ and $l_s \in (2_*,2^*]$, then for any
  $0<\beta_1<\beta_2 \le \infty$ there is $W(\beta_2,\beta_1)>0$ such
  that $V(r,\beta_2)>V(r,\beta_1)$ for $r>W(\beta_2,\beta_1)$ and
  $V(r,\beta_2)=V(r,\beta_1)$, $V'(r,\beta_2)>V'(r,\beta_1)$ for
  $r=W(\beta_2,\beta_1)$.
\end{lemma}
\begin{proof}
  Assume $\boldsymbol{A^-},\boldsymbol{Gu},\boldsymbol{Gs}$ with $l_u
  \ge 2^*$ and $l_s \in [2^*,\sigma^*)$. Then from Proposition
  \ref{super} we know that $U(r,\alpha)$ is a G.S. with s.d.  and that
  it is a S.G.S. with s.d. for $\alpha=\infty$.  Observe that
  $U(r,\alpha_2)>U(r,\alpha_1)$ for $r$ in a right neighborhood of
  $0$, since $U(0,\alpha_2)=\alpha_2>\alpha_1=U(0,\alpha_1)$, and they
  are continuous functions in $r$.  Since they are slow decay
  solutions, from Remark \ref{macchina} we see that there is $R>0$
  (depending on $\alpha_1$, $\alpha_2$) such that
  $U(R,\alpha_2)=U(R,\alpha_1)$. Then we denote by
  \begin{equation}\label{defZ}
    Z(\alpha_2,\alpha_1):= \min \{R>0 \mid U(R,\alpha_2)=U(R,\alpha_1)\}
  \end{equation}
  Thus by construction we get $\frac{\partial}{\partial
    r}U(R,\alpha_2)\le \frac{\partial}{\partial r} U(R,\alpha_1)$ for
  $R=Z(\alpha_2,\alpha_1)$, but from the uniqueness of the solution of
  Cauchy problem for ODEs we see that the inequality is actually
  strict.

  The case $\boldsymbol{A^+},\boldsymbol{Gu},\boldsymbol{Gs}$ with
  $\sigma_* <l_u \le 2^*$ and $l_s \in (2_*,2^*]$ is completely
  analogous, and its proof can be obtained repeating the argument for
  fast decay solutions $V(r, \beta)$ and reversing the direction of
  $s$.
\end{proof}

\subsection{Proof of Theorem \ref{unstable}.}
In this subsection we assume the hypotheses of Theorem \ref{unstable}
without further mentioning.  Let us set
\begin{equation}\label{defsub}
  \begin{split}
    & \psi(x)= \left\{\begin{array}{cc}
        U(|x|,\alpha_2) & \textrm{if $|x|\le Z(\alpha_2,\alpha_1)$} \\
        U(|x|,\alpha_1) & \textrm{if $|x|\ge Z(\alpha_2,\alpha_1)$}
      \end{array}\right.  \\
    & \zeta(x)= \left\{\begin{array}{cc}
        U(|x|,\alpha_2) & \textrm{if $|x|\ge Z(\alpha_2,\alpha_1)$} \\
        U(|x|,\alpha_1) & \textrm{if $|x|\le Z(\alpha_2,\alpha_1)$}
      \end{array}\right.
  \end{split}
\end{equation}
where $Z(\alpha_2,\alpha_1)$ is defined in (\ref{defZ}). Then by
construction $\zeta(x)$ and $\psi(x)$ are radial $C_B$-mild upper and
lower solutions for (\ref{laplace}).  From
Theorem~\ref{thm:theorem-2.4} and Remark~\ref{limit} it follows that
$u(x,t; \zeta)$ is radial, decreasing in $t$ and converges uniformly
to a radial non-negative solution $u(x,T_{\zeta};\zeta)$ of
(\ref{laplace}) for $t <T_{\zeta}$.  Thus $\|u(x,t;
\zeta)(1+|x|^{m(l_s)})\|_{\infty}$ is bounded for $t<T_{\zeta}$, hence
$T_{\zeta}=+\infty$. Since the null solution is the unique radial
solution of (\ref{laplace}) staying below $\zeta(x)$ for any $x \in
\RR^n$, see Lemma \ref{nosol-}, then $\lim_{t \to \infty}
\|u(x,t;\zeta)(1+|x|^{\nu})\|_{\infty}=0$ for any $0\le \nu < m(l_s)$.   \\
Now let $\phi \in C_B$ such that there is $\alpha_2>0$ and $\phi(x)
\lneqq U(|x| , \alpha_2)$. From strong maximum principle for parabolic
equations (see, e.g. the appendix in \cite{Gazzola}), we get
$u(x,t;\phi) < U(|x| , \alpha_2)$ for any $x$ and any $t>0$. So, up to
a time translation, we can assume $\phi(x) < U(|x| , \alpha_2)$ for
any $x \in \RR^n$. From Lemma \ref{uniforme} we see that for any
$\ep>0$ we can find $\alpha_1<\alpha_2$ such that $|U(|x| ,
\alpha_1)-U(|x| , \alpha_2)|<\ep$ for any $|x| \in \RR^n$.  Let
$\zeta(x)$ be the upper solution defined by (\ref{defsub}), then if
$\ep>0$ is small enough we can assume $\phi(x)<\zeta(x)$ for any $x
\in \RR^n$. Hence $0<u(x,t;\phi)<u(x,t;\zeta)$ for any $x \in \RR^n$
and any $t>0$; so $T_{\phi}=+\infty$ and $\|u(x,t;\phi)\|_{\infty} \to
0$, $\|u(x,t;\phi)(1+|x|^{\nu})\|_{\infty} \to 0$ as $t \to \infty$,
for any $0 \le \nu <m(l_s)$.

Similarly consider $\phi \in C_B$ and assume that there is
$\alpha_1>0$ such that $\phi(x)\ge U(|x|,\alpha_1)$, and $\phi(x)
\not\equiv U(|x|,\alpha_1)$.  Reasoning as above we can assume
$\phi(x)> U(|x|,\alpha_1)$ for any $x \in \RR^n$, and we can find
$\alpha_2>\alpha_1$, with $\alpha_2-\alpha_1$ small enough so that the
lower solution $\psi(x)$ defined in (\ref{defsub}) satisfies $\psi(x)<
\phi(x)$ for any $x \in \RR^n$.  From Theorem~\ref{thm:theorem-2.4}
and Remark \ref{limit} it follows that $u(x,t; \psi)$ is radial and
increasing in $t$ and converges uniformly to a radial non-negative
solution $U(x)$ of (\ref{laplace}) for $t <T_{\psi}$ and $U(x) \ge
\psi(x)$ for any $x$. But from Lemma \ref{nosol-} we see that such a
solution $U(x)$ does not exist, hence $T_{\psi}< \infty$ and $\lim_{t
  \to T_{\psi}} \|u(x,t; \psi)\|_{\infty}=+\infty$.  \qed

\subsection{Proof of Theorems \ref{unstableslow} and
\ref{unstableadd1}.}
The construction of the family of upper and lower solutions of
Theorems \ref{unstableslow} and \ref{unstableadd1} is based on Remark
\ref{connection},
Lemmas \ref{picture}, \ref{picture2}.\\
\emph{Proof of Theorem \ref{unstableadd1}.}  Assume first
$\boldsymbol{A^-}, \boldsymbol{Gu},\boldsymbol{Gs}$, with $l_u \ge
2^*$ and $l_s>2^*$. We recall that the critical point
$(\boldsymbol{P^{-\infty}},0)$ of (\ref{si.naa}) admits a
$1$-dimensional unstable manifold and that we denote by
$(\boldsymbol{y^u}(s;l_u),z(s))$ the unique trajectory belonging to
this manifold, and by $U(r,\infty)$ the corresponding solution of
(\ref{radsta}) which is a S.G.S. with slow decay, and by
$\boldsymbol{y^u}(s;l_s)$ the corresponding trajectory of
(\ref{si.na}) with $l=l_s$.  From Lemma \ref{picture} we know that for
any $\tau \in\RR$, $W^u(\tau;l_s)$ and $\boldsymbol{y^u}(\tau;l_s)$
are contained in $\bar{E}^s(\tau)$, i.e.  the bounded set enclosed by
$W^s(\tau;l_s)$ and the $y_2$
coordinate axis (see the construction just before Lemma \ref{picture}).\\
Observe that $\boldsymbol{y^u}(s;l_s) \to \boldsymbol{P^{+\infty}}$,
hence the values
$$
\bar{R} (l_u):=\min \{ y^u_1(s;l_u) \mid s \le 0 \} \, \qquad \bar{R}
(l_s):=\min \{ y^u_1(s;l_s) \mid s \ge 0 \}
$$
are strictly positive and bounded. Assume first that $l_u \le l_s$, so
that $m(l_u) \ge m(l_s)$, and set
\begin{equation}\label{erre}
  \tilde{R}:=1/2\min\big(\{ \bar{R} (l_u) \textrm{exp}[(m(l_s)-m(l_u))\tau] \mid \tau \le 0 \} \cup \{\bar{R} (l_s)\} \big)
\end{equation}
Note that $\tilde{R}>0$ and  $y^u_1(\tau;l_s)>\tilde{R}$ for any $s \in \RR$.\\
If $l_u>2^*$ then $W^u(\tau;l_s)$ is a path connecting the origin and
$\boldsymbol{y^u}(\tau;l_s)$ so we can find (at least) a point
$\boldsymbol{Q^{u,*}}(\tau)= (Q^{u,*}_1(\tau),Q^{u,*}_2(\tau)) \in
W^u(\tau;l_s)$ such that $Q^{u,*}_1(\tau)=\tilde{R}$.  Moreover there
are two points, say $\boldsymbol{Q^{s,+}}(\tau)$,
$\boldsymbol{Q^{s,-}}(\tau)$, belonging to $W^s(\tau;l_s)$ and such
that $\boldsymbol{Q^{s,\pm}}(\tau)=(\tilde{R},Q^{s,\pm}_2(\tau))$, and
$Q^-_2(\tau)<\bar{P}^*_2(\tau)<Q^+_2(\tau)$.  Let us consider now the
trajectories $\boldsymbol{y}(s,\tau;\boldsymbol{Q^{s,+}}(\tau);l_s)$
and $\boldsymbol{y}(s,\tau;\boldsymbol{Q^{s,-}}(\tau);l_s)$: by
construction they correspond to fast decay solutions of
(\ref{radsta}), say $V(r,\beta_2)$ and $V(r,\beta_1)$ respectively.
We can assume w.l.o.g. that $\beta_2>\beta_1$, see Lemma
\ref{corrispondenze} and Remark \ref{corr}.  We denote by
$U(r,\alpha^*)$ the regular solution of (\ref{radsta}) corresponding
to $\boldsymbol{y}(s,\tau;\boldsymbol{Q^{u,*}}(\tau);l_s)$.  Since
$y_2(s,\tau;\boldsymbol{Q^{s,-}}(\tau);l_s)<y_2(s,\tau;\boldsymbol{Q^{u,*}}(\tau);l_s)<
y_2(s,\tau;\boldsymbol{Q^{s,+}}(\tau);l_s)$ we have
$V'(R,\beta_1)<U'(R,\alpha^*)<V'(R,\beta_2)$ for $R=\eu^{\tau}$.
Therefore we can construct upper and lower radial solution of
(\ref{laplace}) , say $\zeta(x)$ and $\psi(x)$ as follows:
\begin{equation}\label{defsubrel}
  \begin{split}
    & \zeta(x,\tau)= \left\{\begin{array}{cc}
        U(|x|,\alpha^*) & \textrm{if $|x|\le \eu^{\tau}$} \\
        V(|x|,\beta_1) & \textrm{if $|x|\ge \eu^{\tau}$}
      \end{array}\right.  \\
    & \psi(x,\tau)= \left\{\begin{array}{cc}
        U(|x|,\alpha^*) & \textrm{if $|x|\ge \eu^{\tau}$} \\
        V(|x|,\beta_2) & \textrm{if $|x|\le \eu^{\tau}$}
      \end{array}\right.
  \end{split}
\end{equation}
Moreover observe that $\beta_i \to +\infty$ and $\alpha^* \to 0$ as
$\tau \to +\infty$, and $\beta_i \to 0$ and $\alpha^* \to +\infty$ as
$\tau \to -\infty$, for $i=1,2$, see Remark \ref{corr}.  Hence
$\zeta(0,\tau)=\psi(0,\tau):=D(\tau) \to 0$ as $\tau \to +\infty$ and
$D(\tau) \to +\infty$ as $\tau \to -\infty$.  Similarly $\lim_{|x| \to
  +\infty}\zeta(x,\tau)|x|^{n-2}=\beta_1$ and $\lim_{|x| \to
  +\infty}\psi(x,\tau)|x|^{n-2}=\beta_2$ go to $0$ as $\tau \to
+\infty$ and they go to $+\infty$ as $\tau \to 0$.

So, from Theorem~\ref{thm:theorem-2.4}, Remark~\ref{limit}, Lemma
\ref{nosol-} we see that for any $\tau \in\RR$ $\|u(x,t;
\zeta(x))\|_{\infty}\to 0$ and $\|u(x,t;
\zeta(x))(1+|x|^{\nu})\|_{\infty}\to 0$ as $t \to +\infty$ for any
$\nu \in [0,n-2)$, and that there is $T_{\psi}$ such that $\|u(x,t;
\psi(x))\|_{\infty}\to +\infty$ as $t \to T_{\psi}$.
\\
Now we go back to the case where $2^*<l_s<l_u$ so that
$m(l_s)>m(l_u)$, and $\boldsymbol{A^-}$ holds. In this case we need to
consider
\begin{equation}\label{erre2}
  \hat{R}:=1/2\min\big(\{ \bar{R} (l_s)
  \textrm{exp}[(m(l_u)-m(l_s))\tau] \mid \tau \ge 0 \} \cup \{\bar{R} (l_u)\} \big)
\end{equation}
and to redefine a set $\hat{E}^s(\tau):= \{ \boldsymbol{Q}=\R
\eu^{(m(l_u)-m(l_s))\tau} \mid \R \in \bar{E}^s(\tau) \}$ (note that
we could redefine $\hat{E}^s(\tau)$ simply considering the bounded set
enclosed by $W^s(\tau;l_u)$ and the $y_2$ axis), and observe that
$W^u(l_u;\tau)\subset \hat{E}^s(\tau)$ for any $\tau \in\RR$.  Then we
repeat the previous argument working with (\ref{si.na}) with $l=l_u$
and replacing $\tilde{R}$ by $\hat{R}$.

Now assume $\boldsymbol{A^-}$, with $l_u = 2^*$ and $l_s>2^*$. In this
case $W^u(\tau;l_s)$ is a $C^1$ manifold departing from the origin and
contained in $\bar{E}^s(\tau)$ for any $\tau \in \RR$, but it does not
have $\boldsymbol{y^u}(\tau;l_s)$ in its border. However for any $\tau
\in \RR$ we can still find at least an intersection between the line
$y_1=\tilde{R}$ and $W^{u}(\tau;l_s)$ and at least two intersection
between $y_1=\tilde{R}$ and $W^{s}(\tau;l_s)$.  So it is easy to check
that the whole argument can be repeated with no changes.

Now assume $\boldsymbol{A^-}$, with $l_u =l_s =2^*$. In this case we
rely on Lemma \ref{picture2}: let $U(r,\alpha^*)$ be the regular
solution of (\ref{radsta}) corresponding to $\boldsymbol{y}(s,\tau,
\boldsymbol{Q^{u,+}}(\tau);2^*)$, and $V(r,\beta_2)$, $U(r,\beta_1)$
be the fast decay solutions of (\ref{radsta}) corresponding
respectively to $\boldsymbol{y}(s,\tau,
\boldsymbol{Q^{s,+}}(\tau);2^*)$ and $\boldsymbol{y}(s,\tau,
\boldsymbol{Q^{s,-}}(\tau);2^*)$.  Then we define $\zeta(x)$ and
$\psi(x)$ as in (\ref{defsubrel}): they are respectively radial upper
and lower solutions for (\ref{laplace}). Moreover
$\zeta(0,\tau)=\psi(0,\tau):=D(\tau) \to 0$ as $\tau \to +\infty$ and
$D(\tau) \to +\infty$ as $\tau \to -\infty$.  Similarly $\lim_{|x| \to
  +\infty}\zeta(x,\tau)|x|^{n-2}=\beta_1$ and $\lim_{|x| \to
  +\infty}\psi(x,\tau)|x|^{n-2}=\beta_2$ go to $0$ as $\tau \to
+\infty$ and they go to $+\infty$ as $\tau \to 0$.

The case where $\boldsymbol{A^+}$, holds with $2_*<l_s \le 2^*$ and
$2_*<l_s<2^*$ or $l_u=l_s=2^*$ are completely analogous and are left
to the reader, so the proofs of Theorem~\ref{unstableadd1} and
Proposition \ref{unstableadd3} is concluded. \qed

\noindent \emph{Proof of Theorem~\ref{unstableslow}.}  Assume
$\boldsymbol{A^-}$, then Theorem~\ref{unstableslow} follows from
Theorem~\ref{unstable}.  So assume $\boldsymbol{A^+}$ and assume first
$l_u,l_s \in (2_*,2^*)$. Then from Proposition \ref{super} there is a
unique S.G.S. with s.d. say $U(r,\infty)$ of (\ref{radsta}): let
$\boldsymbol{y^s}(s;l_u)$ be the corresponding trajectory of
(\ref{si.na}). Fix $\tau \in \RR$: from Lemma \ref{picture} we know
that there is $\boldsymbol{Q}(\tau)=(Q_1(\tau),Q_2(\tau))\in
W^u_{l_u}(\tau)$ such that $Q_1(\tau)=y^s_1(\tau;l_u)$ and
$Q_2(\tau)<y^s_2(\tau;l_u)$. Consider the trajectory
$\boldsymbol{y}(s,\tau, \boldsymbol{Q}(\tau); l_u)$ and the
corresponding regular solution $U(r,d(\tau))$ of (\ref{radsta}): from
Remark \ref{corr} we see that $d(\tau) \to 0$ as $\tau \to +\infty$
and $d(\tau) \to +\infty$ as $\tau \to -\infty$.  Moreover the
function $\chi(x,\tau)$ defined as follows
\begin{equation}\label{defsubslow}
  \begin{split}
    & \chi(x,\tau)= \left\{\begin{array}{cc}
        U(|x|,d(\tau)) & \textrm{if $|x|\le \eu^{\tau}$} \\
        U(|x|,\infty) & \textrm{if $|x|\ge \eu^{\tau}$}
      \end{array}\right.
  \end{split}
\end{equation}
is a super-solution of (\ref{laplace}), and it is regular and has slow
decay. Hence, reasoning as above we see that $u(x,t,\chi)$ is radial,
radially decreasing and well defined for any $t$ and converges
monotonically to the null solution as $t \to +\infty$ and $\lim_{t \to
  +\infty} \|u(x,t,\chi)[1+|x|^{\nu}]\|_{\infty}=0$ for any $0\le \nu
< m(l_s)$.  If $2_*<l_u<l_s=2^*$, we just lose uniqueness of the
S.G.S. with s.d. but the argument can be repeated for any such
solution, so it still works.  If $l_u=l_s=2^*$ we simply need to
repeat the same argument but using Lemma \ref{picture2} instead of
Lemma \ref{picture}. \qed

\end{document}